\documentclass[12pt]{article}

\usepackage[margin=1in]{geometry} % for margins
\usepackage{inputenc} % for french accents
\usepackage{authblk} % author's affiliation
\usepackage{hyperref} % clickable references
\usepackage{color}
\usepackage{array,multirow}
\usepackage{subcaption}
\usepackage{graphicx}
\usepackage{tabularx}
\usepackage{mathtools,amsthm,amsmath,amssymb}
\usepackage{bm} % for bold maths
\usepackage[algo2e,ruled]{algorithm2e} 
\usepackage[capitalise]{cleveref}

\theoremstyle{definition}
\newtheorem{theorem}{Theorem}[section]
\newtheorem{proposition}[theorem]{Proposition}

\newtheorem{assumption}[theorem]{Assumption}

\theoremstyle{definition}
\newtheorem{definition}[theorem]{Definition}
\newtheorem{example}[theorem]{Example}

\theoremstyle{remark}
\newtheorem{remark}[theorem]{Remark}

\newcommand{\R}{\mathbb{R}}
\newcommand{\E}{\mathbb{E}}
\newcommand{\calX}{\mathcal{X}}

\newcommand{\piX}{\pi}
\newcommand{\bX}{{\bm{X}}}
\newcommand{\bx}{{\bm{x}}}

\newcommand{\by}{{\bm{y}}}
\newcommand{\bZ}{{\bm{Z}}}
\newcommand{\bz}{{\bm{z}}}

\newcommand{\bs}{{\bm{s}}}

\DeclareMathOperator{\trace}{trace}

\DeclareMathOperator*{\argmin}{arg\,min}

\renewcommand{\d}{\mathrm{d}}

\usepackage[textsize=tiny,textwidth=1.05in]{todonotes}

\providecommand{\keywords}[1]
{
  \small	
  \textbf{Keywords } #1
}

\begin{document}

\title{Nonlinear dimension reduction for surrogate modeling using gradient information}

\author{\sc 
Daniele Bigoni\footnote{Center for Computational Science \& Engineering, Massachusetts Institute of Technology}, 
Youssef Marzouk$^*$, 
Cl\'ementine Prieur\footnote{Univ. Grenoble Alpes, Inria, CNRS, Grenoble INP, LJK, 38000 Grenoble, France}
and Olivier Zahm$^\dagger$\footnote{Corresponding author (\nolinkurl{olivier.zahm@inria.fr})}
}

\maketitle

\begin{abstract}
 We introduce a method for the nonlinear dimension reduction of a high-dimensional function $u:\R^d\rightarrow\R$, $d\gg1$. Our objective is to identify a nonlinear feature map $g:\R^d\rightarrow\R^m$, with a prescribed intermediate dimension $m\ll d$, so that $u$ can be well approximated by $f\circ g$ for some profile function $f:\R^m\rightarrow\R$. We propose to build the feature map by aligning the Jacobian $\nabla g$ with the gradient $\nabla u$, and we theoretically analyze the properties of the resulting $g$. Once $g$ is built, we construct $f$ by solving a gradient-enhanced least squares problem. Our practical algorithm makes use of a sample $\{\bx^{(i)},u(\bx^{(i)}),\nabla u(\bx^{(i)})\}_{i=1}^N$ and builds both $g$ and $f$ on adaptive downward-closed polynomial spaces, using cross validation to avoid overfitting. We numerically evaluate the performance of our algorithm across different benchmarks, and explore the impact of the intermediate dimension $m$. We show that building a nonlinear feature map $g$ can permit more accurate approximation of $u$ than a linear $g$, for the same input data set.

%  We introduce a method for the nonlinear dimension reduction of a high-dimensional function $u:\mathbb{R}^d\rightarrow\mathbb{R}$, $d\gg1$. Our objective is to identify a nonlinear feature map $g:\mathbb{R}^d\rightarrow\mathbb{R}^m$, with a prescribed intermediate dimension $m\ll d$, so that $u$ can be well approximated by $f\circ g$ for some profile function $f:\mathbb{R}^m\rightarrow\mathbb{R}$. We propose to build the feature map by aligning the Jacobian $\nabla g$ with the gradient $\nabla u$, and we theoretically analyze the properties of the resulting $g$. Once $g$ is built, we construct $f$ by solving a gradient-enhanced least squares problem. Our practical algorithm makes use of a sample $\{x^{(i)},u(x^{(i)}),\nabla u(x^{(i)})\}_{i=1}^N$ and builds both $g$ and $f$ on adaptive downward-closed polynomial spaces, using cross validation to avoid overfitting. We numerically evaluate the performance of our algorithm across different benchmarks, and explore the impact of the intermediate dimension $m$. We show that building a nonlinear feature map $g$ can permit more accurate approximation of $u$ than a linear $g$, for the same input data set.
\end{abstract}


\section{Introduction}

Computational models from a wide range of fields, such as physics, biology, and finance, involve large numbers of uncertain input parameters. Quantifying uncertainty is essential to improving the reliability of these models. Most uncertainty quantification analyses, however, require a large number of model evaluations. When a single evaluation is computationally expensive, a common practice is therefore to replace the model with a \emph{surrogate}---meaning an approximation that can be evaluated cheaply, without further evaluations of the original model. Yet constructing accurate approximations is a challenging task because many function approximation tools become inexpressive in high dimensions. This is often referred as to the \emph{curse of dimensionality}. This problem is exacerbated in the small-data regime, i.e., when few model evaluations are available.

This paper addresses the problem of reducing parameter space dimension from the perspective of surrogate modeling. 
We represent the model by a scalar-valued quantity of interest $u(\bx)$ which depends on a high dimensional parameter $\bx\in\R^d$ with $d\gg1$. When the parameter is uncertain, it is denoted by a random vector $\bX$ whose law models the uncertainty of the parameter. Dimension reduction consists in finding a map $g:\R^d\rightarrow\R^m$, with $m\ll d$, that captures the most ``relevant'' features of the parameters. This \emph{feature map} permits reduction of the parameter dimension from $d$ to $m$ by replacing $\bX$ with the $m$-dimensional random vector ${\bZ} = g(\bX)$. 
From the perspective of surrogate modeling, a good feature map should enable $u(\bX)$ to be well approximated as $f({\bZ}) = f\circ g(\bX)$, for some function $f:\R^m\rightarrow \R$ of $m$ variables only. If such a feature map $g$ is known in advance, $f$ can be constructed by minimizing the mean squared error,
$$
 \E\left[ ( u(\bX) - f\circ g(\bX) )^2 \right] ,
$$
over a class of functions of $m\ll d$ variables. 
This task is, in principle, easier than constructing a $d$-dimensional approximation to $u(\bX)$ directly.

\emph{Linear dimension reduction} corresponds to identifying linear feature maps $g$. Many linear dimension reduction strategies have been proposed in different research fields. Global sensitivity analysis \cite{saltelli2008global} identifies a set of $m$ parameters $g(\bx)=(x_{\sigma_1},\hdots,x_{\sigma_m})$ that best explain, in some statistical sense, the model output.
More generally, ridge functions \cite{pinkus2015ridge} are functions of the form $\bx\mapsto f\circ g(\bx)$ where $g(\bx)=W^T\bx$ for some matrix $W\in\R^{d\times m}$. In \cite{cohen2012capturing,fornasier2012learning}, the model $u$ is assumed to be a ridge function and $W$ is recovered via adaptive model query strategies.
Linear dimension reduction also arises in the statistical regression literature under the name \emph{sufficient dimension reduction} \cite{Adragni09,li2018sufficient}, where $W$ is constructed via sliced inverse regression (SIR) \cite{li1991sliced}, sliced average variance estimation (SAVE) \cite{cookweis1991}, and their variants. Closely related to the present work is the active subspace method \cite{Constantine2014a,constantine2015,lam2020multifidelity}, which identifies $W$ using gradients of the model. The recent papers \cite{zahm2020gradient,parente2020generalized} show that the active subspace method constructs the matrix $W$ by minimizing an upper bound for the mean squared error $\E[ ( u(\bX) - f\circ g(\bX) )^2 ]$ optained with the optimal profile function.
% that one can obtain after constructing $f$.
This result is particularly relevant because it motivates the construction of $g$ from the perspective of approximating $u$ in the least-squares sense.
Similar ideas are developped in \cite{zahm2018certified,cui2020data,brennan2020greedy} for the detection of informed subspace in the context of Bayesian inverse problems.

While linear dimension reduction methods are quite successful in many applications, they can fail to detect certain kinds of low-dimensional structure that a model might have; consider an isotropic $u(\bx)=h(\|\bx\|)$, for instance. 
\emph{Nonlinear dimension reduction} allows $g$ to detect such nonlinear features, in order to improve the approximation power of the composed approximation $f\circ g$.
Nonlinear dimension reduction methods have been developed and analyzed mostly in the community of sufficient dimension reduction; see for instance \cite{wu2008kernel,yeh2008nonlinear,lee2013general}, to cite just a few. 
In these works, the main idea is to use kernel methods to construct a nonlinear feature map $g(\bx)=W^T\Phi(\bx)$, where $\Phi(\bx)=(\Phi_1(\bx),\Phi_2(\bx),\hdots)$ are the eigenfunctions of an ad hoc kernel (typically a squared exponential or a polynomial kernel), and where the matrix $W$ is determined using inverse regression techniques (SIR, SAVE) on the transformed variables ${\bZ} = \Phi(\bX)$. Those methods, however, typically require a large sample size to accurately detect the low-dimensional structure of the model, and thus are not well suited to the small-data regime.
In the spirit of kernel principal component analysis (KPCA), \cite{lataniotis2020extending} builds a feature map of the form $g(\bx)=(\Phi_1(\bx),\hdots,\Phi_m(\bx))$ by taking the $m$ first eigenfunctions of a kernel whose hyperparameters (e.g., correlation length, smoothness) are determined by an outer optimization procedure. 

\subsection{Contribution}
The main contribution of this paper is to propose and analyze a nonlinear parameter space dimension reduction method, for the purpose of function approximation, using gradients of the model.
We assume here that the implementation of the computational model permits computing the gradient of $\bx\mapsto u(\bx)$ with respect to the parameters $\bx$.
Recent advances in computational science permit computing such gradients at a complexity comparable to that of evaluating the model itself, for instance using automatic differentiation \cite{griewank1989automatic} and/or adjoint state methods \cite{plessix2006review}.
Having access to gradient evaluations is a valuable workaround in small-data regimes, as $\nabla u(\bX)$ constitutes additional information for learning the model; see \cite{laurent2019overview}.
In this paper we propose to build $g$ by minimizing the loss function
$$
 J(g) = \E\left[ \big\| \nabla u(\bX) - \Pi_{\mathrm{range}(\nabla g(\bX)^T)} \nabla u(\bX) \big\|^2 \right] ,
$$
where $\Pi_{\mathrm{range}(\nabla g(\bX)^T)}$ denotes the orthogonal projector onto the range of the Jacobian $\nabla g(\bX)^T$.
Intuitively, minimizing this loss yields a feature map whose Jacobian $\nabla g(\bX)$ tends to be aligned with the gradient $\nabla u(\bX)$. 
Based on the same heuristic, the authors of \cite{zhang2019learning} introduce a different loss function to align $\nabla g(\bX)$ with $\nabla u(\bX)$ (see Appendix \ref{appA} for more details) but without proposing a deeper mathematical or computational analysis.
In the present paper, we prove that, under some assumptions, the loss $J(g)$ yields an upper bound on the mean squared error that can be obtained after constructing $f$; that is
$$
 \min_{f:\R^m\rightarrow\R}\E[(u(\bX)-f\circ g(\bX))^2] \leq \mathbb{C} ~ J(g) , 
$$
for some Poincar\'e-type constant $\mathbb{C}$ associated with $\bX$. We propose a quasi-Newton algorithm to minimize $J(g)$ and show that this algorithm is similar to the power iteration used to compute an eigendecomposition in the active subspace method.

In practice, we make use of a data set
$$
 \{{\bx}^{(i)},u({\bx}^{(i)}),\nabla u({\bx}^{(i)})\}_{i=1}^N ,
$$
to estimate the loss $J(g)$ and the mean squared error $\E[ ( u(\bX) - f\circ g(\bX) )^2 ]$.
We assume that the computational cost is dominated by the $N$ evaluations of $u({\bx}^{(i)})$ and $\nabla u({\bx}^{(i)})$, such that the cost for constructing $f$ and $g$ is relatively negligible. 
Borrowing ideas from \cite{beck2013sparsity,migliorati2015adaptive,cohen2018multivariate}, we represent both $f$ and $g$ on adaptive downward-closed polynomial spaces which are built using a greedy algorithm. In order to avoid overfitting, a cross validation procedure is used to determine when to stop the adaptive polynomial enrichment. 
We show that building a nonlinear feature map $g$ permits more accurate approximation of $u$ than a linear $g$, for the same input data set.

We emphasize that our method is a two step procedure: we first build the feature map $g$ by minimizing $J(g)$, and we then build $f$ by minimizing the mean squared error $\E[(u(\bX)-f\circ g(\bX))^2]$.
Another strategy would consist of minimizing the mean squared error \emph{jointly} over $f$ and $g$. For instance, in \cite{hokanson2018data} the authors build a linear $g$ and polynomial $f$ by employing dedicated optimization algorithms on Grassmann manifolds, without using gradients of the model. 
Nonlinear $g$ are also built in \cite{lataniotis2020extending} by joint minimization over $f$ and $g$. However, the structure of such optimization problems, and of the algorithms they employ, remain not well understood.

The rest of this paper is organized as follows. In Section \ref{sec:feature} we analyze the problem of approximating a function $u$ by a composition $f\circ g$. In particular, we give sufficient conditions on $\nabla g$ and $\nabla u$ so that there exists an $f$ such that $f\circ g = u$. We then introduce the loss $J(g)$ and describe its properties regarding the approximation problem. In Section \ref{sec:algos} we present algorithms for constructing $g$ and $f$ on adaptive polynomial spaces. Then, in Section \ref{sec:Numerics}, we illustrate the method on numerical examples.

\section{Dimension reduction via smooth feature maps}\label{sec:feature}

\subsection{Problem statement}

Let $u:\calX\rightarrow\R$ be a scalar-valued function defined on an open set $\calX\subseteq\R^d$ with $d\gg1$.
Our goal is to construct a feature map $g:\calX\rightarrow\R^m$ with $m\ll d$ such that, given a prescribed tolerance $\varepsilon>0$, there exists a function $f:\R^m\rightarrow\R$ for which 
\begin{equation}\label{eq:ErrorLessThanEpsilon}
\E\left[ ( u(\bX)-f(g(\bX)) )^2 \right] \leq \varepsilon^2.
\end{equation} 
Here, $\bX$ denotes a random vector with probability density function $\piX$ such that $\text{supp}(\piX)=\calX$, and $\E[\cdot]$ denotes the mathematical expectation. 
The function $f$ is called the profile function and $m$ the intermediate dimension.
The construction of the profile function is postponed to Section \ref{sec:AdaptiveProfile}, and we focus here on how to find a suitable feature map $g$ such that \eqref{eq:ErrorLessThanEpsilon} is attainable for some $f$.
We note that the $f$ which minimizes the above mean squared error is the conditional expectation $f:{\bz} \mapsto \E[u(\bX)|g(\bX)={\bz}]$. This well-known result will be used later.
We now give two trivial solutions to \eqref{eq:ErrorLessThanEpsilon} which help to understand the problem:
\begin{itemize}
 \item With $g=\text{Id}$, the identity function on $\calX$, the profile function $f=u$ yields $f\circ g = u$. In this case we have $m=d$.
 \item With $g=u$, the profile $f=\text{Id}$ also yields $f\circ g = u$ with an intermediate dimension $m=1$.
\end{itemize}
Those two trivial solutions are not satisfactory either because $m=d\gg1$ is large or because the computation of $g=u$ is untractable. The balance between the intermediate dimension $m$ and the complexity of the feature map $g$ appears as a central question in dimension reduction.
Our goal is to construct $g$ in a tractable space $\mathcal{G}_m$ of functions from $\calX$ to $\R^m$.
For instance, $\mathcal{G}_m $ could be a space of multivariate polynomial functions, a reproducing kernel Hilbert space, etc. We emphasize the necessity of \emph{constraining} the function $g$ to belong to a space of tractable functions; otherwise problem \eqref{eq:ErrorLessThanEpsilon} makes no sense, as it admits a trivial solution with $g=u$.

\subsection{Aligned gradients}\label{sec:AlignedGradients}

From now on, we assume that $u:\calX\rightarrow\R$ is continuously differentiable over the open set $\calX\subseteq \R^d$ and that all the functions in $\mathcal{G}_m$ are also continuously differentiable.
\begin{assumption}\label{assu:RegularityOfUandG}
 $u\in C^1(\calX;\R)$ and $g\in\mathcal{G}_m\subseteq C^1(\calX;\R^m)$. 
\end{assumption}
Let us assume for a moment that $u$ is exactly of the form $u=f\circ g$ for some $g:\calX\rightarrow\R^m$ and $f:\R^m\rightarrow\R$. Denoting by $\nabla f({\bz})\in\R^{m}$ the gradient of $f$ at point ${\bz}\in\R^m$, and by
$$
 \nabla g(\bx) = \begin{pmatrix}
                \nabla g_1(\bx)^T \\ \vdots \\ \nabla g_m(\bx)^T 
               \end{pmatrix}
 \in\R^{m\times d},
$$
the Jacobian\footnote{We use the standard convention that each row of the Jacobian matrix is the transpose of the gradient of each component.} of $g$ at point $\bx\in\calX$, the chain rule allows writing $ \nabla u(\bx) = \nabla g(\bx)^T \nabla f(g(\bx))$ for any $\bx\in\calX$. In this case, $\nabla u(\bx)$ lies in the subspace $\mathrm{range}(\nabla g(\bx)^T)$ for any $\bx\in\calX$. In short, we have
$$
 u=f\circ g 
 \quad\Longrightarrow\quad 
 \nabla u(\bx)\in \mathrm{range}(\nabla g(\bx)^T), \quad \forall\bx\in\calX.
$$
Conversely, one can ask whether a function $u$ which satisfies $\nabla u(\bx) \in \mathrm{range}(\nabla g(\bx)^T)$ for some vector-valued differentiable function $g$ is necessarily of the form of $u=f\circ g$ for some $f$. The following proposition gives a positive answer to this question, under additional assumptions on $g$.

\begin{assumption}\label{assu:pathConnectedG}
The pre-image under $g$ of any point is \textit{smoothly pathwise-connected}; that is, for any ${\bz}\in\text{Im}(g) \subseteq \mathbb{R}^m$ and for any points $\bx, \by$ in the preimage $g^{-1}({\bz})=\{\bs\in\calX:g(\bs)={\bz}\}$, 
 there exists a continuously differentiable function $\gamma:[0,1]\rightarrow g^{-1}({\bz})$ such that $\gamma(0)=\bx$ and $\gamma(1)={\by}$.
\end{assumption}

\begin{proposition}\label{prop:NablaUinNablaG}
 Under Assumptions \ref{assu:RegularityOfUandG} and \ref{assu:pathConnectedG}, if $u:\calX\rightarrow\R$ and $g:\calX\rightarrow\R^m$ satisfy
 \begin{equation}\label{eq:NablaUinNablaG}
  \nabla u(\bx) \in \mathrm{range}(\nabla g(\bx)^T) ,
 \end{equation}
 for any $\bx\in\calX$, then $u=f\circ g$ for some function $f:\R^m\rightarrow\R$.
 
\end{proposition}

\begin{proof}
 We first show that relation \eqref{eq:NablaUinNablaG} implies the following property: if $g(\bx)=g({\by})$ for some $\bx, \by \in\calX$, then $u(\bx)=u({\by})$.
 Thus, let $\bx,{\by}\in\calX$ be any two points such that $g(\bx)=g({\by})$. By Assumption \ref{assu:pathConnectedG}, the pre-image $g^{-1}(\bz), \bz=g(\bx)$, is smoothly pathwise-connected so that there exsits a continuously differentiable path $\gamma:[0,1]\rightarrow\calX$ from $\bx=\gamma(0)$ to ${\by}=\gamma(1)$ such that $g(\gamma(t))=\bz$ for any $t\in[0,1]$. For any $1\leq i \leq m$ the function $g_i\circ\gamma:[0,1]\rightarrow\R$ is constant so that $(g_i\circ\gamma)'(t) = \nabla g_i(\gamma(t))^T\gamma'(t)=0$ for any $t\in[0,1]$, where $\gamma'(t)\in\R^d$ denotes the derivative of $\gamma$ at point $t$. This means that, for any $t\in[0,1]$, the vector $\gamma'(t)$ is orthogonal to $\text{span}\{\nabla g_1(\gamma(t)),\hdots,\nabla g_m(\gamma(t))\}=\mathrm{range}(\nabla g(\gamma(t))^T)$. By \eqref{eq:NablaUinNablaG} we then have 
 $$
  (u\circ\gamma)'(t) = \nabla u(\gamma(t))^T \gamma'(t) = 0,
 $$
 which implies that the continuous function $u\circ \gamma:[0,1]\rightarrow\R$ is constant. Then $u(\bx)=u(\gamma(0))=u(\gamma(1))=u({\by})$.
 
 Now we build a function $f:\R^m\rightarrow\R$ such that $u=f\circ g$. Such a function needs to be defined only on the image $g(\calX)\subseteq\R^m$ and can be set to zero on the complement of $g(\calX)$ in $\R^m$. We define $f$ such that for any ${\bz}\in g(\calX)$, $f({\bz})=u(\bx)$ where $\bx\in\calX$ is any point such that $g(\bx)={\bz}$. Even if this $\bx$ is not unique, $f({\bz})$ is uniquely defined because $u(\bx)=u({\by})$ whenever $g({\by})=g(\bx)$. By construction we have $f(g(\bx))=u(\bx)$ for any $\bx\in\calX$, which concludes the proof.
\end{proof}

\begin{figure}[t]
  \centering 
  \includegraphics[width = 0.8\textwidth]{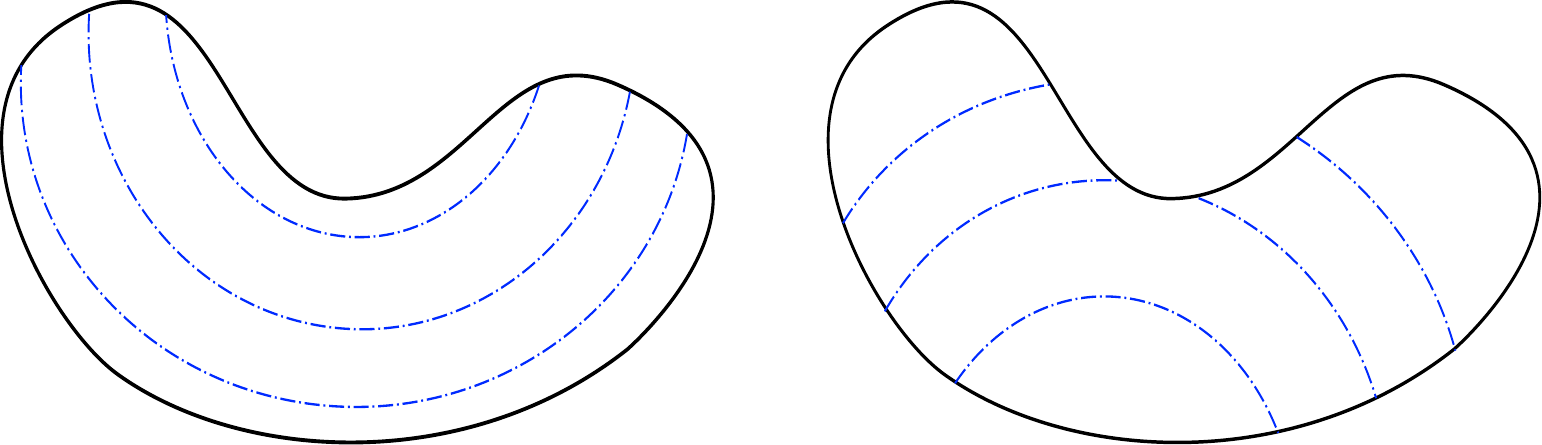} 
  \caption{Illustration of Assumption \ref{assu:pathConnectedG}: the black line represents the parameter space $\calX$ and the blue lines represent different pre-images of two candidate functions $g$. On the left, the function $g$ satisfies Assumption \ref{assu:pathConnectedG}, but on the right, the level sets of the function $g$ are not pathwise-connected.}
  \label{fig:path_connected}
\end{figure}

Let us note that Assumption \ref{assu:pathConnectedG} is a necessary condition in Proposition \ref{prop:NablaUinNablaG}. Indeed, if the pre-images of $g$ are not smoothly pathwise-connected, as in the right plot of Figure \ref{fig:path_connected}, one can build a function $u$ which satisfies \eqref{eq:NablaUinNablaG} without being of the form $f\circ g$. For example, think of a smooth function $u$ which is constant on each of the connected parts of $g^{-1}({\bz})$ (so that \eqref{eq:NablaUinNablaG} is satisfied) but which takes different values on each of those connected parts (so that $u\neq f\circ g$).

Here are some examples where Assumption \eqref{assu:pathConnectedG} is satisfied.

\begin{example}[Affine feature map]\label{example:AffineMaps}
 Any function $g(\bx) = A\bx+b$ with $A\in\R^{m\times d}$ and $b\in\R^m$ satisfies Assumption \ref{assu:pathConnectedG}, provided $\calX$ is a convex set. Indeed, for any ${\bz}\in\R^m$, $\bx,{\by}\in g^{-1}({\bz})$, and $t\in[0,1]$, the quantity $\gamma(t) \coloneqq t\bx+(1-t){\by}$ belongs to $\calX$ and it satisfies  $g(\gamma(t))=t(A\bx+b)+(1-t)(A{\by}+b) ={\bz}$, which shows that $\gamma$ is a continuously differentiable path in $g^{-1}({\bz})$ from $\bx$ to ${\by}$.
\end{example}

\begin{example}[Feature map following from a $C^1$-diffeomorphism]\label{example:C1diffeo}
 Assume $\calX$ is convex.
 One way to build functions which satisfy Assumption \ref{assu:pathConnectedG} is to consider a $C^1$-diffeomorphism $\phi:\calX\rightarrow\calX$, meaning a continuously differentiable invertible function whose inverse is continuously differentiable, and to define $g(\bx) = (\phi_1(\bx),\hdots,\phi_m(\bx))$ where $\phi_i(\bx)$ is the $i$-th component of $\phi(\bx)$. Such a $g$ satisfies Assumption \ref{assu:pathConnectedG}: for any $\bx,{\by}\in\calX$ such that $g(\bx)=g({\by})={\bz}$, the function 
 $$
  \gamma(t) = \phi^{-1}\big( t\phi({\by}) + (1-t)\phi(\bx)\big),
 $$
 defined for $t\in[0,1]$ is a smooth path from $\bx=\gamma(0)$ to ${\by}=\gamma(1)$ as a composition of smooth functions. It is well defined because $t\phi({\by}) + (1-t)\phi(\bx)$ is in $\calX$ by convexity. 
 By construction we have $\phi(\gamma(t)) = t\phi({\by}) + (1-t)\phi(\bx)$ and the $m$ first components of that relation yield $g(\gamma(t)) = tg({\by}) + (1-t)g(\bx) = {\bz}$. This shows that $\gamma(t) \in g^{-1}({\bz})$, so that $g$ satisfies Assumption \ref{assu:pathConnectedG}.
\end{example}

\begin{example}[Polynomial feature map]
 Consider the case where $g$ is a polynomial function on $\calX=\R^d$. Assumption \ref{assu:pathConnectedG} is satisfied if and only if for any ${\bz}\in g(\calX)$, the zeros of the polynomial $x\mapsto g(x)-z$ are pathwise-connected. Calculating the number of connected components (i.e., the zeroth Betti number) of an algebraic set like $\{\bx:g(\bx)-{\bz}=0\}$ is a difficult question, commonly encountered in algebraic geometry. Unfortunately, there is no easy answer to this question; see \cite{scheiblechner2007complexity}. Still, we show later in Section \ref{sec:Numerics} that polynomials work well from a numerical point of view, even though Assumption \ref{assu:pathConnectedG} is not checked in practice.
\end{example}

\subsection{Aligning the gradients}\label{sec:AligningGradients}

Motivated by Proposition \ref{prop:NablaUinNablaG}, we propose to build $g$ by minimizing a cost function which measures how ``aligned'' are the gradient $\nabla u(\bx)$ and the subspace $\mathrm{range}(\nabla g(\bx)^T)$. For any $g:\calX\rightarrow\R^m$ we introduce the cost function
\begin{equation}\label{eq:J}
 J(g) = \E\left[ \big\| \nabla u(\bX) - \Pi_{\mathrm{range}(\nabla g(\bX)^T)} \nabla u(\bX) \big\|^2 \right] ,
\end{equation}
where $\Pi_{\mathrm{range}(\nabla g(\bX)^T)} \in \R^{d\times d}$ denotes the orthogonal projector onto $\mathrm{range}(\nabla g(\bX)^T)$ and $\|\cdot\|$ is the Euclidean norm on $\mathbb{R}^d$. Obviously we have $J(g)\geq0$. The following proposition shows that if $J(g)=0$ then there exists a profile function $f$ such that $u=f\circ g$.
\begin{proposition}\label{prop:wellsuited}
 Let $u:\calX\rightarrow\R$ and $g:\calX\rightarrow\R^m$ be continuously differentiable functions such that $J(g)=0$.
 If $g$ satisfies Assumption \ref{assu:pathConnectedG} and if 
 \begin{equation}\label{eq:gNonDegenerates}
  \mathrm{rank}(\nabla g(\bx)^T)=m,
 \end{equation}
 for any $\bx\in\calX$, then there exists a function $f : \R^m \rightarrow \R$ such that $u=f\circ g$. 
\end{proposition}

Before we give the proof of Proposition \ref{prop:wellsuited}, let us comment on condition \eqref{eq:gNonDegenerates}. 
This condition is commonly encountered in implicit function theory. It ensures that, for all ${\bz}\in g(\calX)$, the level set  $g^{-1}({\bz})$ is a smooth manifold of dimension $d-m$; see for instance Theorem 4.3.1 in \cite{krantz2012implicit}. One can easily check that \eqref{eq:gNonDegenerates} is satisfied in the case of affine feature maps $g(\bx)=A\bx+b$ with $\mathrm{rank}(A)=m$, but also in the case of feature maps following  from a $C^1$-diffeomorphism; see Example \ref{example:C1diffeo}.

\begin{proof}[Proof of \cref{prop:wellsuited}]
 Let us assume for a moment that $\bx\mapsto\Pi_{\mathrm{range}(\nabla g(\bx)^T)}$ is a continuous function from $\calX$ to $\R^{d\times d}$. 
 Then $\bx\mapsto \| \nabla u(\bx) - \Pi_{\mathrm{range}(\nabla g(\bx)^T)} \nabla u(\bx) \|$ is a continuous function, via products and sums of continuous functions.
 As $J(g)=0$, then $ \big\| \nabla u(\bx) - \Pi_{\mathrm{range}(\nabla g(\bx)^T)} \nabla u(\bx) \big\|$ is equal to zero $\piX$-almost surely. By continuity, we have that $\| \nabla u(\bx) - \Pi_{\mathrm{range}(\nabla g(\bx)^T)} \nabla u(\bx) \|$ is equal to zero for all $\bx\in\text{supp}(\piX)=\calX$, so that $\nabla u(\bx) \in \mathrm{range}(\nabla g(\bx)^T)$ holds for any $\bx\in\calX$. 
 Together with Assumption \ref{assu:pathConnectedG}, Proposition \ref{prop:NablaUinNablaG} ensures the existence of $f:\R^m\rightarrow\R$ such that $u=f\circ g$.

 It remains to show that $\bx\mapsto\Pi_{\mathrm{range}(\nabla g(\bx)^T)}$ is continuous. 
 Let $M(\bx)=\nabla g(\bx)\nabla g(\bx)^T\in\R^{m\times m}$. 
 By Assumption \eqref{eq:gNonDegenerates} $M(\bx)$ is invertible and we can write $\Pi_{\mathrm{range}(\nabla g(\bx)^T)} = \nabla g(\bx)^T M(\bx)^{-1} \nabla g (\bx)$ for any $\bx\in\calX$. 
 For any $\delta \in \R^d$ we can write
 \begin{align*}
  \|M(\bx)^{-1}-M(\bx+\delta)^{-1}\|_\text{sp}
  &\leq \|M(\bx+\delta)^{-1}\|_\text{sp}\|M(\bx+\delta)M(\bx)^{-1}-I_d\|_\text{sp}\\
  & = \displaystyle \lambda_{\min}(M(\bx+\delta))^{-1}\|M(\bx+\delta)M(\bx)^{-1}-I_d\|_\text{sp},
 \end{align*}
 where $\|\cdot\|_\text{sp}$ denotes the spectral norm and where $\lambda_{\min}(M(\bx+\delta))$ denotes the smallest eigenvalue of $M(\bx+\delta)$. Because the eigenvalues are continuous with respect to the matrix entries (see \cite{stewart1990matrix}) and by Assumption \eqref{eq:gNonDegenerates}, we have $\lambda_{\min}(M(\bx+\delta))\rightarrow \lambda_{\min}(M(\bx))>0$ as $\delta\rightarrow0$. Therefore we have $\lambda_{\min}(M(\bx+\delta))^{-1}\|M(\bx+\delta)M(\bx)^{-1}-I_d\|_\text{sp} \rightarrow \lambda_{\min}(M(\bx))^{-1}\|I_d-I_d\|_\text{sp}=0$. This shows the continuity of $\bx\mapsto M(\bx)^{-1}$ and therefore the continuity of $\bx\mapsto\Pi_{\mathrm{range}(\nabla g(\bx)^T)}=\nabla g(\bx)^T M(\bx)^{-1} \nabla g (\bx)$. This concludes the proof.
\end{proof}

Next we consider the minimization problem
\begin{equation}\label{eq:minJ}
 \min_{g\in\mathcal{G}_m} J(g),
\end{equation}
where $\mathcal{G}_m\subseteq C^1(\calX;\R^m)$ is a set of tractable functions. 
In general, given some choice of $\mathcal{G}_m$,
the minimum of the cost function will not be exactly zero, and thus an assumption of Proposition \ref{prop:wellsuited} will not hold.
Using arguments based on Poincar\'e inequalities, Proposition \ref{prop:PoincareConditional} below shows that, under specific assumptions, there exists at least one function $f:\R^m\rightarrow\R$ such that $\E[ ( u(\bX)-f(g(\bX)) )^2 ]$ is of the same order of magnitude as $J(g)$. 
In other words, we will be able to control the $L^2$-error in an approximation of $u$ by making $J(g)$ small.
Let us first introduce the Poincar\'e inequality associated with a random variable.

\begin{definition}[Poincar\'e inequality]
 Given a continuous random variable $\bX_\mathcal{M}$ taking values in a smooth manifold $\mathcal{M}$, the \emph{Poincar\'e constant} $\mathbb{C}(\bX_\mathcal{M})\in[0,+\infty]$ is defined as the smallest constant such that
 \begin{equation}\label{eq:PoincareInequality}
  \E\left[ \big( h(\bX_\mathcal{M}) - \E[h(\bX_\mathcal{M})] \big)^2 \right] \leq \mathbb{C}(\bX_\mathcal{M}) \, \E\left[ \big\| \nabla h(\bX_\mathcal{M}) \big\|^2\right] 
 \end{equation}
 holds for any continuously differentiable function $h:\mathcal{M}\rightarrow\R$. 
 Here, the gradient $\nabla h({\bz})$ is a vector in $T_{\bz}(\mathcal{M})$, the tangent space of $\mathcal{M}$ at point ${\bz}\in\mathcal{M}$.
 We say that $\bX_\mathcal{M}$ satisfies the \emph{Poincar\'e inequality} \eqref{eq:PoincareInequality} if $\mathbb{C}(\bX_\mathcal{M})<+\infty$.
\end{definition}

We refer to \cite{bakry2008simple} for a simple proof of the Poincar\'e inequality for a large class of probability measures.

\begin{proposition}\label{prop:PoincareConditional}
 
 Assume that the set of functions $\mathcal{G}_m\subseteq C^1(\calX;\R^m)$ is such that $\mathrm{rank}(\nabla g(\bx)^T)=m$ for all $g\in\mathcal{G}_m$ and all $\bx\in\calX$. Furthermore, assume that $\mathcal{G}_m$ satisfies
 \begin{equation}\label{eq:PoincareConstantConditional}
  \mathbb{C}(\bX|\mathcal{G}_m) \coloneqq \sup_{g \in \mathcal{G}_m}~ \sup_{{\bz}\in g(\calX)} \mathbb{C}(\bX \, | \, g(\bX)={\bz}) ~ <\infty,
 \end{equation}
 where $\bX \, \vert \, g(\bX)={\bz}$ denotes the random variable obtained by conditioning $\bX$ on the event $g(\bX)={\bz}$.
 Then, for any $g\in\mathcal{G}_m$, there exists a measurable $f:\R^m\rightarrow\R$ such that
  \begin{equation}\label{eq:PoincareConditional}
 \E\left[ \big( u(\bX)-f(g(\bX)) \big)^2 \right] \leq \mathbb{C}(\bX|\mathcal{G}_m) J(g),
 \end{equation}
 where $J(g)$ is defined as in \eqref{eq:J}.
\end{proposition}

\begin{proof}[Proof of \cref{prop:PoincareConditional}]
 Let $g\in\mathcal{G}_m$.
 Because $\mathrm{rank}(\nabla g(\bx)^T)=m$ for any $\bx\in\calX$, the level set $\mathcal{M}= g^{-1}({\bz})$ for some ${\bz}\in g(\calX)$ is a smooth manifold of dimension $d-m$; see Theorem 4.3.1 in \cite{krantz2012implicit}. Let $u_\mathcal{M}:\mathcal{M}\rightarrow\R$ be the restriction of $u$ to $\mathcal{M}$. 
 Together with \eqref{eq:PoincareConstantConditional}, the Poincar\'e inequality \eqref{eq:PoincareInequality} with $h=u_\mathcal{M}$ and $\bX_{\mathcal{M}}=(\bX|g(\bX)={\bz})$ permits writing
 \begin{align}
 \E[ ( u(\bX_{\mathcal{M}}) - \E[ u(\bX_{\mathcal{M}})] )^2 ]
 &=\E[ ( u_\mathcal{M}(\bX_{\mathcal{M}}) - \E[ u_\mathcal{M}(\bX_{\mathcal{M}})] )^2 ] \nonumber\\
  &\overset{\eqref{eq:PoincareInequality}\&\eqref{eq:PoincareConstantConditional}}{\leq} \mathbb{C}(\bX|\mathcal{G}_m) \, \E[ \| \nabla u_{\mathcal{M}}(\bX_{\mathcal{M}}) \|^2]. \label{eq:tmp352046}
 \end{align}
 Because $\mathcal{M}$ is a smooth manifold embedded in $\R^d$, the gradient $\nabla u_\mathcal{M}$ can be expressed by means of the gradient $\nabla u$ as follows
 \begin{equation}\label{eq:tmp134}
  \nabla u_\mathcal{M}(\bx) = \Pi_{T_{\bx}(\mathcal{M})} \nabla u(\bx) 
 \end{equation}
 for all $\bx\in\mathcal{M}$, where $\Pi_{T_{\bx}(\mathcal{M})} \in\R^{d\times d}$ is the orthogonal projector onto $T_{\bx}(\mathcal{M})$, the tangent space of $\mathcal{M}$ at $\bx$.
 Since $\mathcal{M}$ is a level set of $g$, we have $T_\bx(\mathcal{M}) = \mathrm{ker}(\nabla g(\bx)) = (\mathrm{range}(\nabla g(\bx)^T))^\perp$ (see for instance \cite[Section 3.5.7]{absil2009optimization}) so that 
 \begin{equation}\label{eq:tmp135}
  \Pi_{T_{\bx}(\mathcal{M})} = \Pi_{\mathrm{ker}(\nabla g(\bx))} = I_d - \Pi_{\mathrm{range}(\nabla g(\bx)^T)}.
 \end{equation}
 Combining \eqref{eq:tmp352046} with \eqref{eq:tmp134} and \eqref{eq:tmp135} we obtain
 \begin{equation}\label{eq:tmp4675}
  \E[ ( u(\bX_{\mathcal{M}}) - \E[ u(\bX_{\mathcal{M}})] )^2 ] \leq \mathbb{C}(\bX|\mathcal{G}_m) ~ \E[ \| (I_d - \Pi_{\mathrm{range}(\nabla g(\bX_{\mathcal{M}})^T)}) \nabla u(\bX_{\mathcal{M}}) \|^2] .
 \end{equation}
 Now, because $\bX_{\mathcal{M}}$ is the conditional random variable $\bX|g(\bX)={\bz}$, we can interpret any expectation $\E[ \phi( \bX_{\mathcal{M}} ) ]$ as a conditional expectation $\E[ \phi( \bX ) | g(\bX)={\bz} ] $ for any integrable function $\phi:\calX\rightarrow\R$. This manipulation permits rewriting the inequality \eqref{eq:tmp4675} as
 \begin{align*}
  \E[ (u(\bX)  -  &\E[u(\bX)|g(\bX)] )^2 \,|\, g(\bX)={\bz} ] \nonumber\\
  &\leq \mathbb{C}(\bX|\mathcal{G}_m)~\E \left[ \| (I_d - \Pi_{\mathrm{range}(\nabla g(\bX)^T)}) \nabla u(\bX) \|^2 \,\Big|\, g(\bX)={\bz} \right] \label{eq:tmp3520410}
 \end{align*}
 Replacing $\bz$ by the random variable $\bZ=g(\bX)$ and taking the expectation on both sides, we obtain
 $$
  \E\left[ (u(\bX)  -  \E[u(\bX)|g(\bX)] )^2 \right] 
  \leq  \mathbb{C}(\bX|\mathcal{G}_m)~\E\left[ \| (I_d - \Pi_{\mathrm{range}(\nabla g(\bX)^T)}) \nabla u(\bX) \|^2 \right].
 $$
 Finally we define the measurable function $f:\R^m\rightarrow\R$ such that $f(\bz)=\E[u(\bX)|g(\bX)=\bz]$ for any $\bz\in\R^m$. We can write $\E[u(\bX)|g(\bX)] = f(g(\bX))$ which yields \eqref{eq:PoincareConditional} and concludes the proof.
\end{proof}

Proposition \ref{prop:PoincareConditional} ensures that, for any $g\in\mathcal{G}_m$, there exists a function $f:\R^m\rightarrow\R$  such that the mean squared error between $u$ and $f\circ g$ is bounded by $\mathbb{C}(\bX|\mathcal{G}_m) J(g)$. This remarkable property justifies the use of the cost function $J$ for the construction of $g$. 

\begin{remark}[Linear feature maps and the Gaussian distribution]
 When $X\sim\mathcal{N}(0,I_d)$ is a standard Gaussian random vector and when $\mathcal{G}_m=\{\bx\mapsto U\bx : U\in\R^{m\times d}, UU^T=I_m \}$ contains linear features, the constant $\mathbb{C}(\bX|\mathcal{G}_m)$ is equal to 1.
 Indeed, the level sets $g^{-1}({\bz})$ are affine subspaces and any conditional random variable of the form $\bX|g(\bX)={\bz}$ is Gaussian with identity covariance. 
 Theorem 3.20 in \cite{boucheron2013concentration} ensures that $\mathbb{C}(\bX|g(\bX)={\bz})=1$ for any $g\in\mathcal{G}_m$ and ${\bz}\in g(\calX)$, which yields $\mathbb{C}(\bX|\mathcal{G}_m)=1$.
\end{remark}

We conclude this section with an important property of $J$.
Consider a $\mathcal{C}^1$-diffeomorphism $\phi:\R^m\rightarrow\R^m$. Since $\nabla\phi(\bx)\in\R^{m\times m}$ is invertible for all $\bx\in\calX$, it holds that $\text{range}(\nabla \phi\circ g(\bX)^T) =\text{range}(\nabla g(\bX)^T \nabla\phi(g(\bX))^T) = \text{range}(\nabla g(\bX)^T)$. Thus we have
\begin{equation}\label{eq:J_invariance}
 J(\phi\circ g) = J(g).
\end{equation}
This invariance reflects the following property of our initial dimension reduction problem \eqref{eq:ErrorLessThanEpsilon}: any composed function $f\circ g$ can be written as the composition of $f\circ\phi^{-1}$ with $\phi\circ g$ so that the feature maps $g$ and $\phi\circ g$ are equivalent with regard to the problem \eqref{eq:ErrorLessThanEpsilon}.
The invariance \eqref{eq:J_invariance} offers the possibility to arbitrarily impose the probability law of $g(\bX)$. Indeed, under natural assumptions on $g$, there exists a $\mathcal{C}^1$-diffeomorphism $\phi=\phi_g$ depending on $g$ so that $\phi_g\circ g(\bX)$ follows, for instance, the standard normal distribution $\mathcal{N}(0,I_d)$; see \cite{villani2008optimal}. 
Replacing $g$ by $\bar g=\phi_g\circ g$ yields the same value of $J(\bar g)=J(g)$ with $\bar g(\bX)\sim\mathcal{N}(0,I_d)$. However, constructing $\phi_g$ can be numerically expensive in practice.
A more pragmatic way to exploit \eqref{eq:J_invariance} is simply to consider the affine transformation $\phi_g({\bz})=\text{Cov}(g(\bX))^{-1/2}({\bz} - \E[g(\bX)])$, which ensures that $\phi_g\circ g(\bX)$ is centered with identity covariance. This affine map is readily computable and allows one to normalize the feature map $g$. 
In the following, we will consider the constrained minimization problem
\begin{equation}\label{eq:minJ_constrained}
 \min_{\substack{g\in\mathcal{G}_m \\ \mathbb{E}[g(\bX)] = 0\\ \mathrm{Cov}(g(\bX)) = I_d}} J(g).
\end{equation}
The constraints $\mathbb{E}[g(\bX)] = 0 $ and $ \mathrm{Cov}(g(\bX)) = I_d$ will be useful to stabilize the minimization algorithms, as described in the next section.

\section{Algorithms}\label{sec:algos}

Based on the previous section, an approximation $f\circ g$ of $u$ can be obtained by first minimizing $J(g)$ over some prescribed feature map space $\mathcal{G}_m$, and then by minimizing the mean squared error $\E[(u(\bX)-f\circ g (\bX))^2]$ over $f\in\mathcal{F}_m$. 
In this section we propose adaptive algorithms to construct a feature map space $\mathcal{G}_m$ of the form
\begin{equation}\label{eq:Gm_vectorSpace}
 \mathcal{G}_m = 
 \left\{ g:\bx\mapsto 
 \begin{pmatrix}
  g_1(\bx)\\\vdots\\g_m(\bx)
 \end{pmatrix}
 \text{ where }
 g_i\in\text{span}\{\Phi_1,\hdots,\Phi_K\}
 \right\}
\end{equation}
and a profile function space $\mathcal{F}_m$ of the form
\begin{equation}\label{eq:Fm_vectorSpace}
 \mathcal{F}_m = \text{span}\{\Psi_1,\hdots,\Psi_P\},
\end{equation}
where $\Phi_1,\hdots,\Phi_K$ and $\Psi_1,\hdots,\Psi_P$ are polynomials defined on $\R^d$ and $\R^m$, respectively.
In practice we make use of a sample $\{(\bx^{(i)},u(\bx^{(i)}),\nabla u(\bx^{(i)})\}_{i=1}^N$ of size $N$, which allows estimating $J(g)$ by
\begin{equation}\label{eq:Jhat}
 \widehat J(g) \coloneqq \frac{1}{N}\sum_{i=1}^N \| \nabla u(\bx^{(i)}) - \Pi_{\mathrm{range}(\nabla g(\bx^{(i)})^T)} \nabla u(\bx^{(i)}) \|^2,
\end{equation}
and the mean squared error $\E[(u(\bX)-f\circ g(\bX))^2]$ by
$\frac{1}{N}\sum_{i=1}^N(u(\bx^{(i)})-f\circ g (\bx^{(i)}))^2$.
First we present in Section \ref{sec:Rayleigh} an algorithm for the minimization of $\widehat J(g)$ over a given (fixed) space $\mathcal{G}_m$.
Then in Section \ref{sec:AdaptiveFeatureMaps} we propose a greedy procedure to enrich the space $\mathcal{G}_m$ adaptively. A similar procedure will be presented in Section \ref{sec:AdaptiveProfile} for the construction of the polynomial space $\mathcal{F}_m$. For those adaptive algorithms, a cross-validation error analysis determines when to stop the enrichment procedures, as described in Section \ref{sec:CrossValidation}.

\subsection{Maximizing the expectation of a Rayleigh quotient}\label{sec:Rayleigh}

Assume the basis $\{\Phi_1,\hdots,\Phi_K\}$ of the feature map space \eqref{eq:Gm_vectorSpace} is given, with $K \geq m$. 
We show that minimizing $J(g)$ (or $\widehat J(g)$) over $g\in\mathcal{G}_m$ boils down to the maximization of the expectation of a generalized Rayleigh quotient. We then propose a quasi-Newton algorithm to solve the problem.

With the notation $\Phi(\bx)=(\Phi_1(\bx) , \hdots ,\Phi_K(\bx))\in\R^{K}$, any feature map $g$ in the space $\mathcal{G}_m$ defined by \eqref{eq:Gm_vectorSpace} can be written as 
$$
 g(\bx)=G^T\Phi(\bx),
$$
for some matrix $G\in\R^{K\times m}$.
In order to account for the constraints $\mathbb{E}[g(\bX)] = 0$ and $ \mathrm{Cov}(g(\bX)) = I_d$ in \eqref{eq:minJ_constrained}, we assume that $\E[\Phi(\bX)]=0$ and we impose the constraint that $G$ satisfy
\begin{equation}\label{eq:G_constrains}
 G^T\mathrm{Cov}(\Phi(\bX))G = I_d.
\end{equation}
Assuming the Jacobian $\nabla g(\bX) = G^T \nabla\Phi(\bX)$ has rank $m$ almost surely, the orthogonal projector $\Pi_{\mathrm{range}(\nabla g(\bX)^T)}$ can be expressed as
$$
 \Pi_{\mathrm{range}(\nabla g(\bX)^T)} = \nabla g(\bX)^T \left(\nabla g(\bX) \nabla g(\bX)^T \right)^{-1} \nabla g(\bX),
$$
and the cost function $J(g)$ becomes
\begin{align}
  J(g)
  &=\E\left[ \big\| \nabla u(\bX) - \Pi_{\mathrm{range}(\nabla g(\bX)^T)} \nabla u(\bX) \big\|^2 \right] \nonumber\\
  &= \E\left[ \| \nabla u(\bX) \|^2 \right] - \E\left[ \big\| \Pi_{\mathrm{range}(\nabla g(\bX)^T)} \nabla u(\bX)  \big\|^2\right] \nonumber\\
  &=\E\left[ \| \nabla u(\bX) \|^2\right]- \E\left[ \nabla u(\bX)^T \nabla g(\bX)^T \left(\nabla g(\bX) \nabla g(\bX)^T \right)^{-1} \nabla g(\bX) \nabla u(\bX)  \right] \nonumber\\
  &= \E\left[ \| \nabla u(\bX) \|^2\right]- \E\left[ \trace \big(G^T A(\bX) G \big)\big(G^TB(\bX)G\big)^{-1} \right]. \label{eq:J_expectedRayleighQuotient}\nonumber
\end{align}
Here, $A(\bX)\in\R^{K\times K}$ and $B(\bX)\in\R^{K\times K}$ are two symmetric positive semidefinite matrices given by
\begin{align*}
 A(\bX) &= \nabla \Phi(\bX) \nabla u(\bX) \nabla u(\bX)^T \nabla \Phi(\bX)^T ,\\
 B(\bX) &= \nabla \Phi(\bX) \nabla \Phi(\bX)^T.
\end{align*}
Minimizing $g\mapsto J(g)$ over $\mathcal{G}_m$ is the same as maximizing
\begin{equation}\label{eq:R}
 \mathcal{R}(G) = \E\left[ \trace \left(\big(G^T A(\bX) G \big)\big(G^TB(\bX)G\big)^{-1} \right)\right],
\end{equation}
over $G\in\R^{K\times m}$. 
Similarily, minimizing $g\mapsto \widehat J(g)$ over $\mathcal{G}_m$ is the same as maximizing
\begin{equation}\label{eq:Rhat}
 \widehat{\mathcal{R}}(G) = \frac{1}{N}\sum_{i=1}^N \trace \left(\big(G^T A(\bX^{(i)}) G \big)\big(G^TB(\bX^{(i)})G\big)^{-1} \right),
\end{equation}
over $G\in\R^{K\times m}$. 
The quantity $\mathcal{R}(G)$ corresponds to the expectation of the generalized Rayleigh quotient associated with the matrix pair $(A(\bX),B(\bX))$, and $\widehat{\mathcal{R}}(G)$ to its Monte Carlo estimate. It is easier to recognize the generalized Rayleigh quotient when $m=1$, since $G\in\R^K$ becomes a vector so that $\mathcal{R}(G)=\E[\frac{G^T A(\bX) G}{G^T B(\bX) G}]$ and $\widehat{\mathcal{R}}(G)=\frac{1}{N}\sum_{i=1}^N \frac{G^T A(\bx^{(i)}) G}{G^T B(\bx^{(i)}) G}$.
Generalized Rayleigh quotients are ubiquitous in dimension reduction; see \cite{kokiopoulou2011trace}. However, the \emph{expectations} or \emph{sums} of generalized Rayleigh quotients as in \eqref{eq:R} and \eqref{eq:Rhat} are not common and appear to be much more difficult to maximize.
As shown in \cite{wang2018efficient,zhang2013optimizing,zhang2014self}, maximizing the sum of two generalized Rayleigh quotients is already a difficult task, which requires dedicated algorithms.
In the particular case where the feature map is linear, however, maximizing $\mathcal{R}(G)$ can be done analytically, as shown by the next remark.

\begin{remark}[Linear feature maps and active subspaces]\label{rmk:AS}
 The space of linear feature maps $\mathcal{G}_m = \{\bx\mapsto G^T\bx : G\in\R^{d\times m}\} $ corresponds to \eqref{eq:Gm_vectorSpace} with $\Phi(\bx)=\bx$, the identity map. In this case $\nabla\Phi(\bx)=I_d$ is independent of $\bx$ so that $A(\bX)=\nabla u(\bX)\nabla u(\bX)^T$ and $B(\bX)=I_d$. The expected generalized Rayleigh quotient \eqref{eq:R} becomes the standard (matrix) Rayleigh quotient $\mathcal{R}(G) = \trace((G^T H G)(G^T G)^{-1})$ where 
 $$
  H = \E[\nabla u(\bX)\nabla u(\bX)^T].
 $$
 The maximum of $G\mapsto\mathcal{R}(G)$ is known to be attained by any matrix $G\in\R^{K\times m}$ whose columns span the $m$-dimensional dominant eigenspace of $H$.
 This subspace is sometimes called the active subspace; see \cite{constantine2015,Constantine2014a,zahm2020gradient}.
 When considering the sample approximation $\widehat{\mathcal{R}}(G)$ in \eqref{eq:Rhat}, the matrix $H$ is simply replaced by its approximation $\widehat H = \frac{1}{N}\sum_{i=1}^N \nabla u(\bx^{(i)})\nabla u(\bx^{(i)})^T$.
 The accuracy of the active subspace recovery from $\widehat H$ depends on the sample size $N$, on the active subspace dimension $m$, and on the spectrum of $H$; see \cite{lam2020multifidelity} for more details.
\end{remark}

So far we have seen that, provided the basis $\Phi(\bx)=(\Phi_1(\bx) , \hdots ,\Phi_K(\bx))$ satisfies $\E[\Phi(\bX)]=0$, the minimization problem \eqref{eq:minJ_constrained} can be rewritten as
\begin{equation}\label{eq:maxR}
 \min_{\substack{g\in\mathcal{G}_m \\ \mathbb{E}[g(\bX)] = 0\\ \mathrm{Cov}(g(\bX)) = I_d}} J(g)
 \qquad \overset{g(\bx)=G^T\Phi(\bx)}{\Longleftrightarrow} \qquad 
 \max_{\substack{G\in\R^{K\times m}  \\ G^T\mathrm{Cov}(\Phi(\bX))G = I_d}} \mathcal{R}(G).
\end{equation}
Next we propose a quasi-Newton method to solve this problem. The following proposition gives the expression for the gradient of $G\mapsto\mathcal{R}(G)$. The proof is given in Appendix \ref{proof:gradR}.

\begin{proposition}\label{prop:gradR}
 Let $A(\bX),B(\bX) \in\R^{K\times K}$ be two random symmetric positive semidefinite matrices.
 Assume that for a given $G\in\R^{K\times m}$, there exists $\varepsilon>0$ such that $(G+\delta G)^T B(\bX)(G+\delta G)$ is almost surely invertible for any $\|\delta G\|\leq\varepsilon$.
 Then $\mathcal{R}(\cdot)$ defined by \eqref{eq:R} is differentiable at $G$ and its gradient $\nabla\mathcal{R}(G)\in\R^{K\times m}$ is such that $(\nabla\mathcal{R}(G))_{ij} = \frac{\partial\mathcal{R}(G)}{\partial G_{ij}}$ can be written as
 \begin{align}
  \nabla\mathcal{R}(G) &= 2 \left(\big(  H(G) - \Sigma(G)\big)G_\mathrm{vec} \right)_\mathrm{mat}, \label{eq:gradR}
 \end{align}
 where $H(G)$ and $\Sigma(G)$ are two symmetric positive semidefinite matrices in $\R^{(Km)\times (Km)}$ given by
 \begin{align}
  H(G) &= \E\left[ \left((G^T B(\bX)G )^{-1}\right)  \otimes A(\bX) \right] \label{eq:H}\\
  \Sigma(G) &= \E\left[ \left((G^T B(\bX)G )^{-1} G^T A(\bX)G(G^T B(\bX)G )^{-1}  \right) \otimes B(\bX) \right] \label{eq:Sigma}.
 \end{align}
 Here, the notation $(\cdot)_\mathrm{vec}$ denotes the vectorization of a matrix, such that $G_\mathrm{vec}\in\R^{Km}$ is the vertical concatenation of the columns of $G\in\R^{K\times m}$. The matricization $(\cdot)_\mathrm{mat}$ is the reverse operation, such that $(G_\mathrm{vec})_\mathrm{mat}=G$. 
 The notation $\otimes$ denotes the Kronecker product.
\end{proposition}

Starting at an initial guess $G^{(0)}\in\R^{K\times m}$, a quasi-Newton method for maximizing $G\mapsto\mathcal{R}(G)$ is an iterative procedure $G^{(k+1)} = G^{(k)} - (\mathcal{H}^{(k)})^{-1}\nabla\mathcal{R}(G^{(k)})$ where $\mathcal{H}^{(k)}:\R^{K\times m}\rightarrow\R^{K\times m}$ is an approximation to the Hessian of $\mathcal{R}(\cdot)$ at point $G^{(k)}$; see \cite{dennis1977quasi}.
Because our goal is to maximize $\mathcal{R}(\cdot)$, the operator $\mathcal{H}^{(k)}$ should be chosen symmetric negative definite.
We propose to use $\mathcal{H}^{(k)} = -2\Sigma(G^{(k)})$. 
This matrix naturally appears in the expression of the Hessian $\nabla^2\mathcal{R}(G^{(k)})$ when differentiating the relation \eqref{eq:gradR}.
Assuming $\Sigma(G^{(k)})$ is invertible (we observe in practice that it is non-singular) the quasi-Newton iteration in vectorized form is
\begin{align}
 G_\mathrm{vec}^{(k+1)} 
 &= G_\mathrm{vec}^{(k)} - \left( \big( \mathcal{H}^{(k)} \big)^{-1} \nabla\mathcal{R}(G^{(k)}) \right)_\mathrm{vec} \nonumber\\
 &\overset{\eqref{eq:gradR}}{=} G_\mathrm{vec}^{(k)} - \Big(-2\Sigma(G^{(k)}) \Big)^{-1}\Big( 2H(G^{(k)}) - 2\Sigma(G^{(k)})\Big)G_\mathrm{vec}^{(k)} \nonumber\\
 &= \Sigma(G^{(k)})^{-1}H(G^{(k)})G_\mathrm{vec}^{(k)}. \label{eq:QuasiNewton_first}
\end{align}
To account for the constraint $G^T \text{Cov}(\Phi(\bX)) G = I_m$ in \eqref{eq:maxR}, notice that, by the definition \eqref{eq:R} of $\mathcal{R}(\cdot)$, we have $\mathcal{R}(GM)=\mathcal{R}(G)$ for any invertible matrix $M\in\R^{m\times m}$. 
By letting $M=(G^T\text{Cov}(\Phi(\bX))G)^{-1/2}$, the matrix $\widetilde G = GM$ satisfies the constraint ${\widetilde G}^T \text{Cov}(\Phi(\bX)) {\widetilde G} = I_m$ and yields the same Rayleigh quotient $\mathcal{R}(\widetilde G)=\mathcal{R}(G)$. 
Following this reasoning, we modify the iterations \eqref{eq:QuasiNewton_first} by adding a normalization step:
\begin{align}
 G^{(k+1/2)} &= \left( \Sigma(G^{(k)})^{-1}H(G^{(k)})G_\mathrm{vec}^{(k)}  \right)_\mathrm{mat} ,\label{eq:QuasiNewton}\\
 G^{(k+1)} &= G^{(k+1/2)} \left( G^{(k+1/2)T} \text{Cov}(\Phi(\bX))G^{(k+1/2)} \right)^{-1/2} .\label{eq:QuasiNewtonNormalization}
\end{align}
Interestingly, this quasi-Newton procedure is very similar to a power iteration for solving eigenvalue problems; see the next remark.

\begin{remark}[Quasi-Newton method and power iteration]\label{rmk:PowerIteration}
 Let us continue Remark \ref{rmk:AS}, where $\mathcal{G}_m$ is the space of linear feature maps. Recall that $\Phi(\bx)=\bx$, $A(\bX)=\nabla u(\bX)\nabla u(\bX)^T$, $B(\bX) = I_d$, and assume for simplicity that $\mathrm{Cov}(\Phi(\bX)) = I_d$. Given an iterate $G^{(k)}$ such that  $G^{(k)}G^{(k)T}=I_d$, the matrices $H(G^{(k)})$ and $\Sigma(G^{(k)})$ introduced in \eqref{eq:H} and \eqref{eq:Sigma} become
 $H(G^{(k)}) = I_d\otimes H $ and $\Sigma(G^{(k)}) =( G^{(k)T} H G^{(k)} ) \otimes I_d$,
 where $H=\E[\nabla u(\bX)\nabla u(\bX)^T]$.
 Using the relation $((S_2\otimes S_1)G_\mathrm{vec})_\mathrm{mat} = S_1 G S_2$ for any symmetric matrices $S_1,S_2$, the quasi-Newton iteration \eqref{eq:QuasiNewton} becomes
 \begin{equation}\label{eq:tmp3467865}
  G^{(k+1/2)} = \left(\left( \left( G^{(k)T} H G^{(k)} \right)^{-1} \otimes H\right) G_\mathrm{vec}^{(k)} \right)_\mathrm{mat}= H G^{k} \left( G^{(k)T} H G^{(k)} \right)^{-1}.
 \end{equation}
 Thus, the relation
 $$
  \text{range}(G^{(k+1)}) 
  \overset{\eqref{eq:QuasiNewtonNormalization}}{=} \text{range}(G^{(k+1/2)}) 
  \overset{\eqref{eq:tmp3467865}}{=} \text{range}(HG^{(k)}) 
  = \text{range}(H^{k+1} G^{(0)}) 
 $$
 holds and shows that the quasi-Newton iteration \eqref{eq:QuasiNewton} with the normalization step \eqref{eq:QuasiNewtonNormalization} is precisely a power iteration method which aims to compute the $m$-dimensional dominant eigenspace of the matrix $H$.
\end{remark}

In practice, the quasi-Newton method \eqref{eq:QuasiNewton} and \eqref{eq:QuasiNewtonNormalization} can be used to maximize $\widehat{\mathcal{R}}(G)$ \eqref{eq:Rhat} by replacing $H(G)$ and $\Sigma(G)$ with their sample approximations:
\begin{align*}
  \widehat H(G) &= \frac{1}{N}\sum_{i=1}^N \left((G^T B(\bx^{(i)})G )^{-1}\right)  \otimes A(\bx^{(i)}) \\
  \widehat \Sigma(G) &= \frac{1}{N}\sum_{i=1}^N \left((G^T B(\bx^{(i)})G )^{-1} G^T A(\bx^{(i)})G(G^T B(\bx^{(i)})G )^{-1}  \right) \otimes B(\bx^{(i)}) .
\end{align*}
The procedure is summarized in Algorithm \ref{alg:QuasiNewton}.
In the next section, we propose a relevant choice for the initialization $G^0$ of Algorithm \ref{alg:QuasiNewton}.
We emphasize that assembling these $Km$-by-$Km$ matrices would require the storage of $K^2m^2$ scalars, which is obviously not affordable when $K$ (and $m$) are large. In  practice, we never assemble these matrices explicitly.
Using the formulas
\begin{align}
 \widehat H(G)x &= \left(\frac{1}{N}\sum_{i=1}^N A(\bx^{(i)}) x_\mathrm{mat} (G^T B(\bx^{(i)})G )^{-1} \right)_\mathrm{vec} \label{eq:Hhat}\\
 \widehat \Sigma(G)x &= \left(\frac{1}{N}\sum_{i=1}^N  B(\bx^{(i)}) x_\mathrm{mat} (G^T B(\bx^{(i)})G )^{-1} G^T A(\bx^{(i)})G(G^T B(\bx^{(i)})G )^{-1} \right)_\mathrm{vec} \label{eq:Sigmahat},
\end{align}
the matrix-vector products $x\mapsto H(G)x$ and $x\mapsto \Sigma(G)x$ are computationally tractable.
In this sense, the matrices $H(G)$ and $\Sigma(G)$ are \emph{implicit} matrices.
For the calculation of $x\mapsto\widehat\Sigma(G)^{-1}x$, as required in \eqref{eq:QuasiNewton}, iterative solvers are well suited because they rely only on matrix-vector products; see \cite{golub2013matrix}. Here we use a conjugate gradient solver preconditioned with the diagonal matrix containing the diagonal of $\widehat\Sigma(G)$.

\begin{algorithm2e}[h]
\caption{Quasi-Newton method to maximize $G\mapsto\widehat{\mathcal{R}}(G)$.\label{alg:QuasiNewton}}
\SetKwInput{KwInput}{Input}                % Set the Input
\SetKwInput{KwOutput}{Output}              % set the Output
\SetKwInput{KwRequire}{Require}
\DontPrintSemicolon
  \KwRequire{Computing the matrix-vector products $x\mapsto \widehat H(G)x$ and $x\mapsto \widehat \Sigma(G)x$ as in \eqref{eq:Hhat} and \eqref{eq:Sigmahat}.}
  \KwData{Training sample}
  \KwInput{Feature map space $\mathcal{G}_m$, initial guess $G^{(0)}\in\R^{K\times m}$, tolerance $\varepsilon>0$, max iteration $K_{\max}$}
  ~\\
  
    Initialize $k= 0$ and $\text{stepsize}=\varepsilon+1$
   
   \While{ $k< K_{\max}$ \rm{and} $\mathrm{stepsize} \geq\varepsilon$ }{
      Compute $b = \widehat H(G^{(k)})G^{(k)}_\mathrm{vec}\in\R^{Km}$ \;
      Solve $\widehat\Sigma(G^{(k)}) x = b$ using preconditioned conjugate gradient \;
      Matricize $x_\mathrm{mat}=(x)_\mathrm{mat}\in\R^{K\times m}$ and update $G^{(k+1/2)} = G^{(k)} - x_\mathrm{mat}$ \;
      Normalize $G^{(k+1)} = G^{(k+1/2)} M^{-1/2}$ with 
      $$M=G^{(k+1/2)T}\text{Cov}(\Phi(\bX))G^{(k+1/2)}\in\R^{m\times m}$$
      Update $k\gets k+1$ and $\text{stepsize} \gets \|x\|$
     }
  \KwOutput{final iterate $G^{(k)}$}
\end{algorithm2e}

\subsection{Adaptive polynomial feature map space}\label{sec:AdaptiveFeatureMaps}

In the previous section we proposed an algorithm for minimizing $g\mapsto \widehat J(g)$ over a given feature map space $\mathcal{G}_m$, as in \eqref{eq:Gm_vectorSpace}.
In this section, we borrow ideas from \cite{beck2013sparsity,migliorati2015adaptive,cohen2018multivariate} to construct $\mathcal{G}_m$ adaptively using multivariate polynomials.

We assume that the probability density function $\piX$ of $\bX$ is a product density $\piX(\bx) = \piX_1(\bx_1)\hdots\piX_d(\bx_d)$. 
For any $1\leq\nu\leq d$ we denote by $\{\Phi_{0}^\nu,\Phi_{1}^\nu,\hdots\}$ an orthonormal polynomial basis, with the degree of $\Phi_i^{\nu}$ equal to $i$, such that
$$
 \int \Phi_{i}^\nu(x)\Phi_{j}^\nu(x) \piX_\nu(x)\d x=\delta_{ij},
$$
holds for any $i,j\geq0$. For any multi-index $\alpha=(\alpha_1,\hdots,\alpha_d)\in\mathbb{N}^d$, we define the multivariate polynomial $\Phi_\alpha$ as
$$
 \Phi_\alpha(\bx)=  \prod_{\nu=1}^d \Phi_{\alpha_\nu}^\nu(\bx_\nu),
$$
and, for a given multi-index set $\Lambda_K\subseteq\mathbb{N}^d$ of cardinality $\#\Lambda_K = K$, we introduce 
\begin{equation}\label{eq:Gm_vectorSpace_LambdaK}
 \mathcal{G}_m^{\Lambda_K} = 
 \left\{ 
 \bx\mapsto 
 \begin{pmatrix}
  g_1(\bx)\\\vdots\\g_m(\bx)
 \end{pmatrix},~
 g_i\in\text{span}\{\Phi_\alpha;\alpha\in\Lambda_K\}
 \right\} .
\end{equation}
This feature map space parametrized by $\Lambda_K$ is, up to a change of notation, of the form of $\mathcal{G}_m$ in \eqref{eq:Gm_vectorSpace}.
The optimal multi-index set $\Lambda_K$ is that which minimizes the minimum of $J(g)$ over $g\in\mathcal{G}_m^{\Lambda_K}$, meaning
\begin{equation}\label{eq:LambdaK_optimal}
 \argmin_{\substack{\Lambda_K \subseteq \mathbb{N}^d \\ \#\Lambda_K=K}}
 \min_{g\in \mathcal{G}_m^{\Lambda_K} }  \widehat J(g).
\end{equation}
This best $K$-term approximation problem is combinatorial and not tractable in practice.
We propose a suboptimal solution to \eqref{eq:LambdaK_optimal} using a greedy procedure of the form
$$
 \Lambda_{K+1} = \Lambda_K \cup \{\alpha_{K+1}\},
$$
where $\alpha_{K+1}\in\mathbb{N}^d$ is a multi-index to determine.
Suppose we are given $\Lambda_K$ and that the corresponding optimal feature map
$$
 g_{\Lambda_K} \in\underset{g\in\mathcal{G}_m^{\Lambda_K}}{\text{argmin }} \widehat J(g)
$$
has been computed (for instance using Algorithm \ref{alg:QuasiNewton}). 
The optimal multi-index $\alpha_{K+1}$ to add would be the one which minimizes $\alpha\mapsto \widehat J(g_{\Lambda_K \cup \{\alpha\}})$. This would require the computation of $g_{\Lambda_K \cup \{\alpha\}}$ for many $\alpha\in\mathbb{N}^d$, which is not affordable in practice.
Instead we choose the multi-index $\alpha_{K+1}$ as the one which yields the steepest gradient of the function $v\mapsto \widehat J(g_{\Lambda_K}+v \Phi_\alpha)$ around $v=0$, meaning
\begin{equation}\label{eq:alphaKplus1}
 \alpha_{K+1} \in \underset{\alpha\in\mathbb{N}^d}{\text{arg\,max }}  \left\|\nabla_v ~\widehat J(g_{\Lambda_K}+v \Phi_\alpha)\big|_{v=0} \right\|.
\end{equation}
The rationale behind \eqref{eq:alphaKplus1} is to select the polynomial $\Phi_\alpha$ which, once added to the feature map space $\mathcal{G}_m$, yields the best immediate improvement of $\widehat J(\cdot)$ when moving away from $g_{\Lambda_K}$ in the direction $\Phi_\alpha$.

Maximization over the entire $\mathbb{N}^d$ as in \eqref{eq:alphaKplus1} is not feasible in practice. A standard workaround is to search for the maximum over an arbitrary subset of $\mathbb{N}^d$ with finite cardinality. The subset $\{\alpha\in\mathbb{N}^d, \sum_{i=1}^d \alpha_i \leq p\}$ is commonly used, as it corresponds to the polynomials $\Phi_\alpha$ with total degree bounded by $p$. However the cardinality of this subset is ${d+p \choose d}=\frac{(d+p)!}{p!d!}$ which can still be very large.
Borrowing ideas from \cite{migliorati2015adaptive,migliorati2019adaptive}, we propose an alternative strategy which relies on the notion of downward-closed sets; see \cite{chkifa2015breaking,cohen2018multivariate}. 
We assume that the set $\Lambda_K$ is downward-closed, meaning that
\begin{equation}\label{eq:downwardClosedLambdaK}
 \alpha \in \Lambda_K \text{ and } \alpha'\leq\alpha ~\Rightarrow~ \alpha' \in \Lambda_K,
\end{equation}
where $\alpha'\leq\alpha $ means $\alpha_i'\leq\alpha_i$ for all $1\leq i \leq d$.
Intuitively, \eqref{eq:downwardClosedLambdaK} means that $\Lambda_K$ has a pyramidal shape that contains no hole.
We denote by $\mathcal{M}(\Lambda_K)$ the \emph{reduced margin} of $\Lambda_K$, defined by
$$
 \mathcal{M}(\Lambda_K)
 = \{ \alpha\in\mathbb{N}^d\backslash\Lambda_K \text{ such that }  \alpha-e_i\in\Lambda_K \text{ for all } 1\leq i \leq d \text{ with } \alpha_i\neq0 \}
$$
where $e_i$ denotes the $i$-th canonical vector of $\mathbb{N}^d$.
By construction, any set of the form $\Lambda_K\cup\{\alpha\}$ with $\alpha\in\mathcal{M}(\Lambda_K)$ remains downward closed, which is the fundamental property of the reduced margin. By searching for the new multi-index in the reduced margin of $\Lambda_K$, as in
$$
 \alpha_{K+1} \in \underset{\alpha\in \mathcal{M}(\Lambda_K)}{\text{argmax }}   \left\|\nabla_v ~\widehat J(g_{\Lambda_K}+v \Phi_\alpha)\big|_{v=0} \right\|,
$$
we ensure that $\Lambda_{K+1}$ remains downward closed. This is illustrated on Figure \ref{fig:GreedyOnReducedMargin}.

\begin{figure}[t]
    \centering
    \begin{subfigure}[t]{0.5\textwidth}
        \centering
        \includegraphics[width=0.5\textwidth]{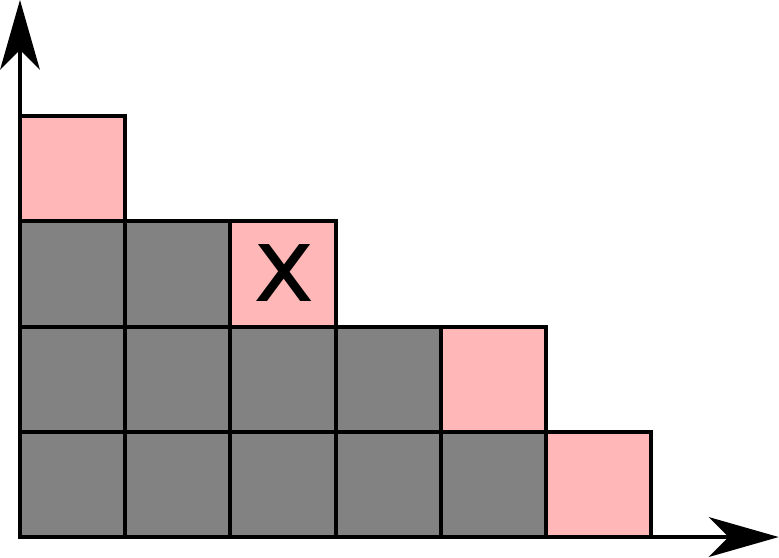}
        \caption{$\Lambda_K$, $\mathcal{M}(\Lambda_K)$ and $\alpha_{K+1}$.}
    \end{subfigure}%
    ~ 
    \begin{subfigure}[t]{0.5\textwidth}
        \centering
        \includegraphics[width=0.5\textwidth]{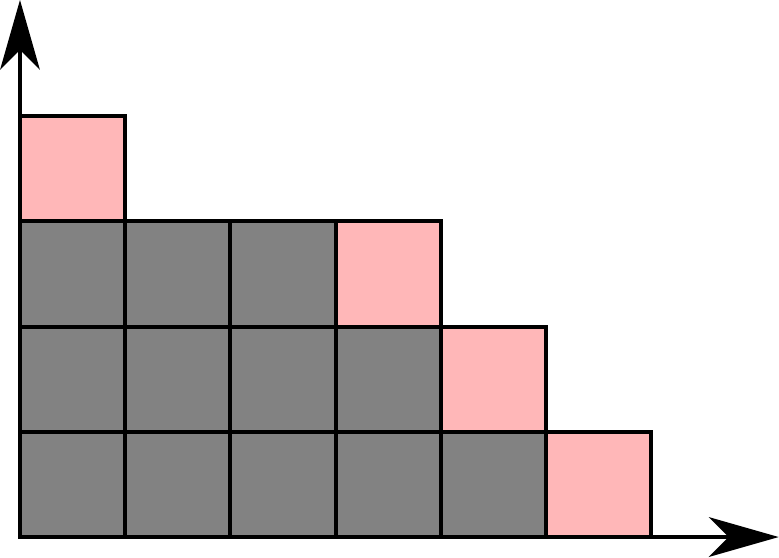}
        \caption{$\Lambda_{K+1}$ and $\mathcal{M}(\Lambda_{K+1})$.}
    \end{subfigure}
    \caption{Greedy construction of the downward closed set $\Lambda_K\subseteq\mathbb{N}^d$ with $d=2$. 
    Adding $\alpha_{K+1}$ (the cross on the left) to $\Lambda_K$ (gray boxes on the left) yields $\Lambda_{K+1}$ and the new reduced margin $\mathcal{M}(\Lambda_{K+1})$ (right plot).}
    \label{fig:GreedyOnReducedMargin}
\end{figure}

As pointed out in \cite{migliorati2015adaptive,cohen2018multivariate} in the context of least-squares regression, adding multiple multi-indices at each greedy iteration could yield better performance compared to adding only one multi-index at a time. Instead of the enrichment $\Lambda_{K+1} = \Lambda_K \cup \{\alpha_{K+1}\}$, we consider the so-called \emph{bulk chasing} procedure
$$
 \Lambda_{K+1} = \Lambda_K \cup \lambda_{K+1},
$$
where $\lambda_{K+1}\subseteq\mathcal{M}(\Lambda_K)$ is the smallest set of multi-indices such that
\begin{equation}\label{eq:lambdaKplus1_Margin}
 \left(\sum_{\alpha\in\lambda_{K+1}}
 \left\|\nabla_v ~\widehat J(g_{\Lambda_K}+v \Phi_\alpha)\big|_{v=0} \right\|^2  \right)
 \geq \theta \left(\sum_{\alpha\in\mathcal{M}(\Lambda_K)}
 \left\|\nabla_v ~\widehat J(g_{\Lambda_K}+v \Phi_\alpha)\big|_{v=0} \right\|^2  \right),
\end{equation}
for some parameter $0<\theta\leq1$. That is, $\lambda_{K+1}$ contains the $\#\lambda_{K+1}$ largest values of $\|\nabla_v ~\widehat J(g_{\Lambda_K}+v \Phi_\alpha)\big|_{v=0} \|$ which capture a prescribed fraction $\theta$ of the norm of the gradient of $J$ on the reduced margin. With the bulk chasing procedure we have $\#\Lambda_{K}\neq K$ in general.

This procedure is summarized in Algorithm \ref{alg:GreedyFeatureMaps}. We choose to start the algorithm with the set $\Lambda_{K}=\Lambda_{d}=\{\alpha\in\mathbb{N}^d:\sum_{i=1}^d \alpha_i=1\}$. This corresponds to the space of linear feature maps and, as explained in Remark \ref{rmk:PowerIteration}, Algorithm \ref{alg:QuasiNewton} boils down to a power iteration for which a random initialization works well.
Later, we initialize Algorithm \ref{alg:QuasiNewton} by adding a row of zeros to $G_{\Lambda_K}$ to account for the newly added basis terms. 
Notice that Algorithm \ref{alg:GreedyFeatureMaps} stops after $K_{\max}$ iterations. We will explain in Section \ref{sec:CrossValidation} how to use cross validation to determine $K_{\max}$.

\begin{algorithm2e}
  \caption{Construction of feature map $g$ on a downward-closed polynomial space\label{alg:GreedyFeatureMaps}}
  \DontPrintSemicolon
  \SetKwInput{KwInput}{Input}                % Set the Input
  \SetKwInput{KwOutput}{Output}              % set the Output
  
  \KwData{Training sample}
  \KwInput{Intermediate dimension $m$, max iteration $K_{\max}$, parameter $\theta$}
  
  Initialize $K=d$ and $\Lambda_{K}=\{\alpha\in\mathbb{N}^d:\sum_{i=1}^d \alpha_i=1\}$ \;
  Compute $G_{\Lambda_K}\in\R^{d\times m}$ using Algorithm \ref{alg:QuasiNewton} with random initialization. \;
  Define $g_{\Lambda_K}(\bx)=G_{\Lambda_K}^T\bx$  \;
  \For{$K=d,\hdots,K_{\max}-1$}{
  Compute $ \|\nabla_v \widehat J(g_{\Lambda_K}+v \Phi_\alpha)|_{v=0} \|$ for all $\alpha\in\mathcal{M}(\Lambda_{K})$ \;
     Select $\lambda_{K+1}$ as in \eqref{eq:lambdaKplus1_Margin}\;
     Update $\Lambda_{K+1} = \Lambda_{K} \cup \lambda_{K+1}$ and $\mathcal{G}_m^{\Lambda_{K+1}}$\;
     Compute $G_{\Lambda_{K+1}} \in \mathcal{G}_m^{\Lambda_{K+1}}$ using Algorithm \ref{alg:QuasiNewton} initialized with
     $$
      G_{\Lambda_{K+1}}^{(0)} = 
      \left[\begin{matrix}
       G_{\Lambda_K} \\ [0,\hdots,0]
      \end{matrix}\right]\in\R^{(K+1)\times m}
     $$
     Define $g_{\Lambda_{K+1}}(\cdot) = G_{\Lambda_{K+1}}^T \Phi(\cdot)$, where $\Phi = [\Phi_1 ,\hdots, \Phi_{\alpha_{K+1}}]:\R^d\rightarrow\R^{K+1}$\;
}
  \KwOutput{final iterate $g_{\Lambda_{K_{\max}}}$}
\end{algorithm2e}

\begin{remark}
 The greedy procedure of Algorithm~\ref{alg:GreedyFeatureMaps} can get stuck because it ``doesn't see'' behind the reduced margin. For instance, if a relevant index is located above $\mathcal{M}(\Lambda_K)$ and if the gradient vanishes on the reduced margin, the algorithm will never activate that index.  \cite{migliorati2019adaptive}  suggests a safeguard mechanism to avoid this: arbitrarily activate the most ancient index from the reduced margin every $n$-th iteration. In our numerical tests, however, we never needed such a safeguard mechanism.
 
\end{remark}

\subsection{Adaptive polynomial profile function space}\label{sec:AdaptiveProfile}

In this section we assume the feature map $g$ has been computed using Algorithm \ref{alg:GreedyFeatureMaps}.
We now build the profile function $f$ in a polynomial space $\mathcal{F}_m$.
As in the previous section, we propose to greedily enrich $\mathcal{F}_m$ so that the minimum of the empirical mean squared error $\widehat{\mathcal{E}}_g (f) = \frac{1}{N}\sum_{i=1}^N (u({\bx}^{(i)})-f\circ g ({\bx}^{(i)}))^2$ over $f\in\mathcal{F}_m$ is minimized.
Since the gradients $u({\bx}^{(1)}),\hdots,u({\bx}^{(N)})$ are available, we instead consider the gradient-enhanced empirical mean squared error,
\begin{equation}\label{eq:RMShat_gradientEnhanced}
 \widehat{\mathcal{E}}_g^\nabla (f) = \frac{1}{N}\sum_{i=1}^N
 \Big( (u({\bx}^{(i)})-f\circ g ({\bx}^{(i)}))^2
 + \| \nabla u({\bx}^{(i)})- \nabla  f\circ g ({\bx}^{(i)}) \|^2
 \Big).
\end{equation}
Using $\widehat{\mathcal{E}}_g^\nabla (f)$ instead of $\widehat{\mathcal{E}}_g(f)$ is known to yield better mean squared error in the small sample regime; see \cite{peng2016polynomial}. 
This will be illustrated in the next section.
Given a finite multi-index set $\Gamma_L\subseteq\mathbb{N}^m$ we introduce
\begin{equation}\label{eq:Fm_Space_LambdaK}
 \mathcal{F}_m^{\Gamma_L} = 
 \text{span}\{\Psi_\alpha;\alpha\in\Gamma_L\},
\end{equation}
where $\Psi_\alpha$ denotes the $\alpha$-th multivariate Hermite polynomial.
These polynomials form an orthogonal basis of $L^2_{\mathcal{N}(0,I_m)}$. 
In the present context it would have been preferable to work with a $L^2_{g_\sharp\mu}$-orthogonal basis, but such a basis is not readily obtainable as it would require computing expensive high-dimensional integrals (e.g., for a Gram-Schmidt procedure).
We justify the use of Hermite basis by the fact that, since $g(\bX)$ is centered and has identity covariance (recall the constraints in \eqref{eq:minJ_constrained}), $\{\Psi_\alpha\}_{\alpha\in\mathbb{N}^d}$ is a relatively well conditioned basis in $L^2_{g_\sharp\mu}$.
We show numerically in Section \ref{sec:Numerics} that Hermite polynomials perform well.

As in the previous section, we propose to build a sub-optimal solution to the best $L$-term approximation problem
$$
 \min_{\substack{\Gamma_L \subseteq \mathbb{N}^d \\ \#\Gamma_L=L}}
 \min_{f\in \mathcal{F}_m^{\Gamma_L} }  \widehat{\mathcal{E}}_g^\nabla (f)
$$
by greedily constructing the multi-index set as follows:
$
 \Gamma_{L+1} = \Gamma_L \cup \lambda_{L+1},
$
where $\lambda_{L+1}\subseteq\mathcal{M}(\Gamma_L)$ is the smallest multi-index set such that
\begin{equation}\label{eq:greedyGamma}
 \left(\sum_{\alpha\in\lambda_{L+1}} \left|
 \frac{\d}{\d t} \widehat{\mathcal{E}}_g^\nabla (f_{\Gamma_L} + t \Psi_{\alpha})\Big|_{t=0} 
 \right|^2
 \right)
 \geq \theta
 \left(
 \sum_{\alpha\in\mathcal{M}(\Gamma_L)} \left|
 \frac{\d}{\d t} \widehat{\mathcal{E}}_g^\nabla (f_{\Gamma_L} + t \Psi_{\alpha})\Big|_{t=0} 
 \right|^2
 \right).
\end{equation}
Here, $f_{\Gamma_L}$ denotes the minimizer of $\widehat{\mathcal{E}}_g^\nabla (f)$ over $f\in\mathcal{F}_m^{\Gamma_L}$ and $\mathcal{M}(\Gamma_L)$ the reduced margin of $\Gamma_L$.
This is summarized in Algorithm \ref{alg:GreedyProfile}.
Since $\widehat{\mathcal{E}}_g^\nabla (f)$ is quadratic in $f$, this algorithm corresponds to an Orthogonal Matching Pursuit (OMP) approach, as explained in the next remark.

\begin{remark}
 Using the expansion $f = \sum_{l=1}^L w_l \Psi_{\alpha_l} \in \mathcal{F}_m^{\Gamma_L}$ with ${\mathbf w}=(w_1, \ldots , w_L)^T \in\R^L$, the gradient-enhanced empirical mean squared error \eqref{eq:RMShat_gradientEnhanced} can be written as
 $
  \widehat{\mathcal{E}}_g^\nabla (f) = \| y - A {\mathbf w}\|^2,
 $
 where $y\in\R^{N(d+1)}$ is given by 
 $$
  y = 
  \frac{1}{\sqrt{N}}\begin{pmatrix}
   u(\bx^{(1)}) & \hdots & u(\bx^{(N)}) \\
   \nabla u(\bx^{(1)}) & \hdots & \nabla u(\bx^{(N)}) \\
  \end{pmatrix}_\text{vec}
 $$
 and the $\alpha$-th column of the matrix $A=[A_{\alpha_1}\cdots A_{\alpha_L}]\in\R^{N(d+1) \times L}$ is
 $$
  A_\alpha = 
  \frac{1}{\sqrt{N}}
  \begin{pmatrix}
   \Psi_\alpha({\bz}^{(1)}) & \hdots & \Psi_\alpha({\bz}^{(N)}) \\
   \nabla g({\bx}^{(1)}) \nabla \Psi_\alpha({\bz}^{(1)}) & \hdots & \nabla g({\bx}^{(N)}) \nabla \Psi_\alpha({\bz}^{(N)}) \\
  \end{pmatrix}_\text{vec}
 $$
 with ${\bz}^{(i)}=g({\bx}^{(i)})$.
 Recall that the subscript ``vec'' stands for the vectorization of a matrix. 
 Thus we have
 $
  |\frac{\d}{\d t} \widehat{\mathcal{E}}_g^\nabla (f + t \Psi_{\alpha})|_{t=0}| = 
 |
  A_\alpha^T ( y-Ax^L )
 |,
 $
 which shows that the selection procedure \eqref{eq:greedyGamma} corresponds to choosing the (nonactive) column of $A_\alpha$ which is most correlated with the residual $y-Ax^L$. This is similar to the OMP algorithm \cite{tropp2007signal}; the difference is that, instead of seeking $\alpha$ in a prescribed set, Algorithm \eqref{eq:greedyGamma} seeks $\alpha$ in $\mathcal{M}(\Gamma_L)$, which evolves during the iteration process.
\end{remark}

\begin{algorithm2e}[t]
\caption{Construction of profile function $f$ on downward-closed polynomial space\label{alg:GreedyProfile}}
\SetKwInput{KwInput}{Input}                % Set the Input
\SetKwInput{KwOutput}{Output}              % set the Output
\SetKwInput{KwRequire}{Require}
\DontPrintSemicolon
  \KwData{Training sample}
  \KwInput{Feature map $g$ with intermediate dimension $m$, max iteration $L_{\max}$, parameter $\theta$}
  
    Initialize $\Gamma_{0}=\{(0,\hdots,0)\}$ \;
    Solve the least-squares problem $f_{\Gamma_0} = \min\{ \widehat{\mathcal{E}}_g^\nabla (f) ;f\in\mathcal{F}_m^{\Gamma_0} \}$ \;
   
   \For{$L=0,\hdots,L_{\max}-1$}{
   Compute $|\frac{\d}{\d t} \widehat{\mathcal{E}}_g^\nabla (f + t \Psi_{\alpha})|_{t=0}|$ for all $\alpha\in\mathcal{M}(\Gamma_{L})$  \; 
   Select $\lambda_{L+1}$ as in \eqref{eq:greedyGamma} \; 
   Update $\Gamma_{L+1} = \Gamma_{L} \cup \lambda_{L+1}$ and $\mathcal{F}_m^{\Gamma_{L+1}}$ \; 
   Solve the least-squares problem $f_{\Gamma_{L+1}} = \min\{ \widehat{\mathcal{E}}_g^\nabla (f) ;f\in\mathcal{F}_m^{\Gamma_{L+1}} \}$ \;
   }
   
  \KwOutput{final iterate $f_{\Gamma_{L_{\max}}}$}
\end{algorithm2e}

\subsection{Cross-validation}\label{sec:CrossValidation}

Algorithms \ref{alg:GreedyFeatureMaps} and \ref{alg:GreedyProfile} need to be stopped before they begin overfitting the data. We employ the $\nu$-fold cross-validation procedure decribed in Algorithm \ref{alg:CV}. It consists of partitioning the initial sample $\Xi = \{({\bx}^{(i)},u({\bx}^{(i)}),\nabla u({\bx}^{(i)})\}_{i=1}^N$ into $\nu$ subsets $\Xi_i^{\textup{train}}$, $i=1, \ldots, \nu$ of equal cardinality $N/\nu$, then running the algorithms on each subset $\Xi_i^{\textup{train}}$ while monitoring the error on the corresponding test set $\Xi_i^{\textup{test}}=\Xi \backslash\Xi^\textup{train}_i$. The optimal number of iterations $K^*$ (for Algorithm \ref{alg:GreedyFeatureMaps}) and $L^*$ (for Algorithm \ref{alg:GreedyProfile}) are those which minimize the test error averaged over the $\nu$ folds. With these numbers in hand, we then run $K^*$ and $L^*$ iterations of the algorithms on the entire sample.

In Algorithm \ref{alg:CV}, we use the same sample to train both $f$ and $g$. Alternatively, we can build $f$ and $g$ using two independent samples. We tried this alternative without obtaining significant improvement. Thus, in the context where the model $u$ is expensive to evaluate, we recommend training $f$ and $g$ on the same sample.

\begin{algorithm2e}[H]
\SetKwInput{KwInput}{Input}                % Set the Input
\SetKwInput{KwOutput}{Output}              % set the Output
\DontPrintSemicolon
  \caption{Learning a composed model $f\circ g \approx u$ using values and gradients of $u$ \label{alg:CV}}
  \KwData{Sample $\{({\bx}^{(i)},u({\bx}^{(i)}),\nabla u({\bx}^{(i)})\}_{i=1}^N$}
  \KwInput{Intermediate dimension $m$, max iteration $K_{\max}$ and $L_{\max}$, number of folds $\nu$\\~}

  \SetKwProg{Fn}{Partition the data set $\Xi = \{({\bx}^{(i)},u({\bx}^{(i)}),\nabla u({\bx}^{(i)})\}_{i=1}^N$ for cross validation}{}{}
  \Fn{}{
    Partition $\Xi$ into $\nu$ subsets of equal cardinality:\;
    \textbf{ - $i$-th test set:}   $\Xi^\textup{test }_i $ is the $i$-th subset of $\Xi$ \;
    \textbf{ - $i$-th training set:} $\Xi^\textup{train}_i = \Xi \backslash\Xi^\textup{test }_i $
  }
  ~\\~
  
  \SetKwProg{Fn}{Construction of the feature map}{}{}
  \Fn{}{
        \For{$i=1,\hdots,\nu$}{
      Run $K_{\max}$ iterations of Algorithm \ref{alg:GreedyFeatureMaps} on the \textbf{$i$-th training set}\;
      Store the iterates $g^{(1)},\hdots,g^{(K_{\max})}$\;
      Monitor the loss $\mathcal{J}_{i,j} = \widehat J(g^{(j)})$, $1\leq j \leq K_{\max}$, on the \textbf{$i$-th test set}\;
    }
    Define $K^*$ as the minimum of the mean $j\mapsto\frac{1}{\nu}\sum_{i=1}^\nu \mathcal{J}_{i,j}$\;
    Run $K^*$ iterations of Algorithm \ref{alg:GreedyFeatureMaps} using the \textbf{whole sample} $\Xi$\;
    \KwRet feature map $g=g^{(K^*)}$\;
  }
  ~\\~
  
  \SetKwProg{Fn}{Construction of the profile}{}{}
  \Fn{}{
        \For{$i=1,\hdots,\nu$}{
      Run $L_{\max}$ iterations of Algorithm \ref{alg:GreedyProfile} on the \textbf{$i$-th training set}\;
      Store the iterates $f^{(1)},\hdots,f^{(L_{\max})}$\;
      Monitor the mean squared error $\mathcal{E}_{i,j} = \widehat{\mathcal{E}}_g(f^{(j)})$, $1\leq j \leq L_{\max}$, on the \textbf{$i$-th test set}\;
    }
    Define $L^*$ as the minimum of the mean $j\mapsto\frac{1}{\nu}\sum_{i=1}^\nu \mathcal{E}_{i,j}$\;
    Run $L^*$ iterations of Algorithm \ref{alg:GreedyProfile} using the \textbf{whole sample} $\Xi$\;
    \KwRet profile function $f=f^{(L^*)}$\;
  }
  ~\\~

  \KwOutput{Composed approximation $f\circ g$}
\end{algorithm2e}

\section{Numerical examples}\label{sec:Numerics}

Source code for the algorithms above and numerical experiments below is freely available\footnote{\texttt{https://gitlab.inria.fr/ozahm/nonlinear-dimension-reduction-for-surrogate-modeling.git}} so that all results presented here are entirely reproducible.
Our implementation uses the toolbox \emph{ApproximationToolbox} \cite{nouy_anthony_2020_3653971}.

\subsection{Isotropic function}

We first consider the function $u:\R^d\rightarrow\R$ with $d=20$ defined by
$$
 u(\bx) =  \cos(\|\bx\|_2),
$$
and we let $\mu=\mathcal{N}(0,I_d)$ be the standard normal distribution. This function is isotropic: it cannot be well approximated by $f\circ g$ with a linear feature map $g$. However, if one allows $g$ to be a quadratic polynomial, the function $g(x)=x_1^2+\hdots+x_{20}^2=\|\bx\|_2^2$ allows one to write $u=f\circ g$ with a rather simple one-dimensional profile function, $f(z)=\cos(\sqrt{z})$. 

First we assess the performance of the quasi-Newton method (Algorithm \ref{alg:QuasiNewton}) for the minimization of $g\mapsto\widehat J(g)$ over a \emph{fixed} space of feature maps $\mathcal{G}_m$. Results are reported in Figure \ref{fig:CosOfNorm_QuasiNewton}. During the first 20 iterations, $\mathcal{G}_m$ is chosen to be the space of linear feature maps; after the 21st iteration, $\mathcal{G}_m$ is enlarged to contain linear and quadratric feature maps. During the first period, we observe a rapid convergence of $J(g)$ towards a plateau which decreases with $m$. Once the quadratic terms are activated, $J(g)$ converges toward zero at an exponential rate. This shows the efficiency of the quasi-Newton approach in Algorithm \ref{alg:QuasiNewton} for building $g$ on a fixed function space $\mathcal{G}_m^{\Lambda_{K}}$. We observe that the convergence rates are not the same for $m=1$, $m=5$, and $m=10$.

\begin{figure}[t]
  \centering 
  \includegraphics[width = 0.55\textwidth]{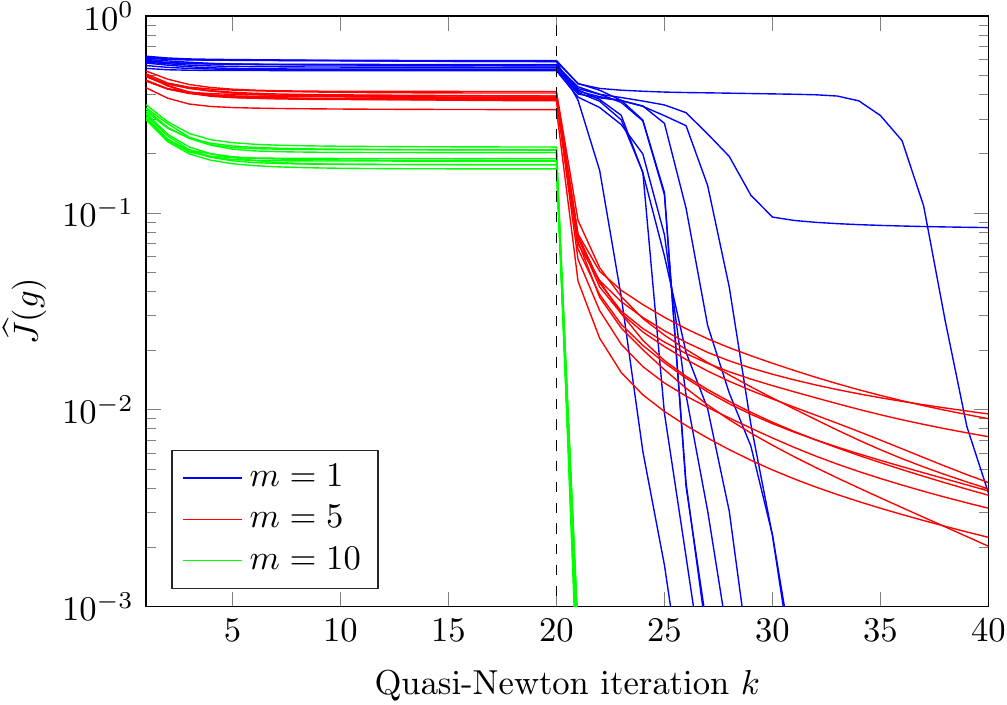} 
  \caption{Isotropic function. Evolution of $\widehat J(g)$ during the quasi-Newton algorithm \ref{alg:QuasiNewton} using $N=100$ gradients of $u$ (10 different realizations). For the first 20 iterations, $\mathcal{G}_m^{\Lambda_{K}}$ contains linear functions only ($\#\Lambda_{K} = 20$). At the 21st iteration, $\mathcal{G}_m^{\Lambda_{K}}$ is enlarged to include all quadratic functions ($\#\Lambda_{K}=20+210=230$).}
  \label{fig:CosOfNorm_QuasiNewton}
\end{figure}

Figure \ref{fig:CosOfNorm_OMP_g} shows the behavior of the \emph{adaptive} Algorithm \ref{alg:GreedyFeatureMaps} for constructing a feature map $g$. Recall that Algorithm \ref{alg:GreedyFeatureMaps} is initialized with $\Lambda_K=\{\alpha\in\mathbb{N}^{20}:\sum_{i=1}^d\alpha_i=1\}$, which corresponds to the space of linear feature maps. For this experiment, we enrich $\Lambda_K$ with only one multi-index at a time,  i.e., $\Lambda_{K+1}=\Lambda_K\cup\{\alpha_{K+1}\}$ with $\alpha_{K+1}$ as in \eqref{eq:alphaKplus1}. We observe that the algorithm is always capable of building a polynomial $g$ such that $J(g)=0$ with very few greedy iterations. Note that for large $m$, $J(g)=0$ is attained earlier, i.e., for smaller $\#\Lambda_K$. To explain this phenomenon, Table \ref{tab:IsotropicFunctionExactDecompositions} lists a few exact decompositions $u=f\circ g$, where we see that a large intermediate dimension $m$ compensates for a small feature map space $\#\Lambda_K$.

Figure \ref{fig:CosOfNorm_OMP_f} shows the performance of Algorithm \ref{alg:GreedyProfile}. We set the bulk chasing parameter to $\theta=0.3$ and we run a cross-validation procedure (Algorithm \ref{alg:CV}) with $\nu=5$ folds to determine when to stop the enrichment process. With $m=1$, the algorithm is capable of recovering a very accurate approximation to $u$ (error below $10^{-4}$) with only $N=100$ samples. In contrast, using the same sample, a full dimensional polynomial approximation (black curves in Figure \ref{fig:CosOfNorm_OMP_f}) can barely attain errors below $10^{-1}$. With intermediate dimensions $m=5$ and $m=10$, we still outperform the full dimensional approach $d=m$,  but the error does not reach $10^{-2}$.
This example nicely illustrates the fundamental issue of balancing the complexity between $f$ and $g$: 
\begin{itemize}
 \item With $m=1$, we obtain a complex $g\in\mathcal{G}_m^{\Lambda_K}$ with $\#\Lambda_K\geq40$ and a simple $f\in\mathcal{F}_m^{\Gamma_L}$ with $\#\Gamma_L\leq5$. Error is below $10^{-4}$.
 \item With $m=5$ or $m=10$, we obtain a simpler $g\in\mathcal{G}_m^{\Lambda_K}$ with $30\leq\#\Lambda_K\leq40$ and a more complex $f\in\mathcal{F}_m^{\Gamma_L}$ with $20\leq\#\Gamma_L\leq100$. 
 Error is around $2 \times 10^{-2}$.
 \item With $m=d$, (no dimension reduction) $g(x)=x$ is linear and $f\in\mathcal{F}_m^{\Gamma_L}$ with $\#\Gamma_L\geq300$. Error barely falls below $10^{-1}$.
\end{itemize}
Clearly, for the considered isotropic function, the optimal choice of intermediate dimension is $m=1$. We will see in the next examples that this is not always the case.

\begin{figure}[t]
  \centering 
  \begin{subfigure}[t]{.4\textwidth}
  \centering
  \includegraphics[width=1\linewidth]{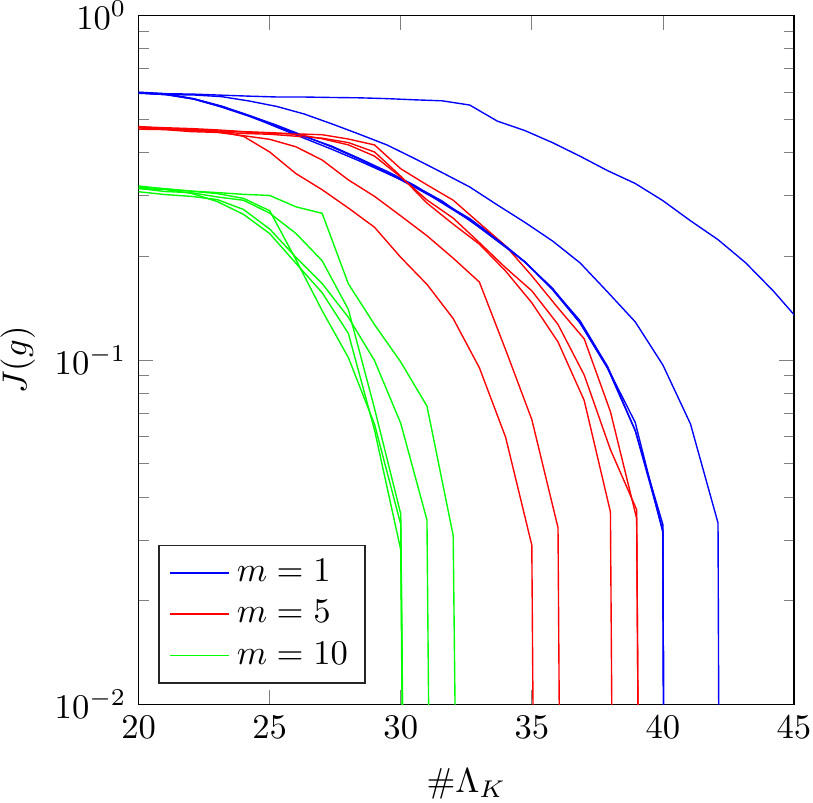} 
  \caption{Evolution of $J(g)$ during the greedy enrichment process of Algorithm  \ref{alg:GreedyFeatureMaps}.}
  \label{fig:CosOfNorm_OMP_g}
\end{subfigure}\hspace{1cm}
\begin{subfigure}[t]{.5\textwidth}
  \centering
  \includegraphics[width=0.8\linewidth]{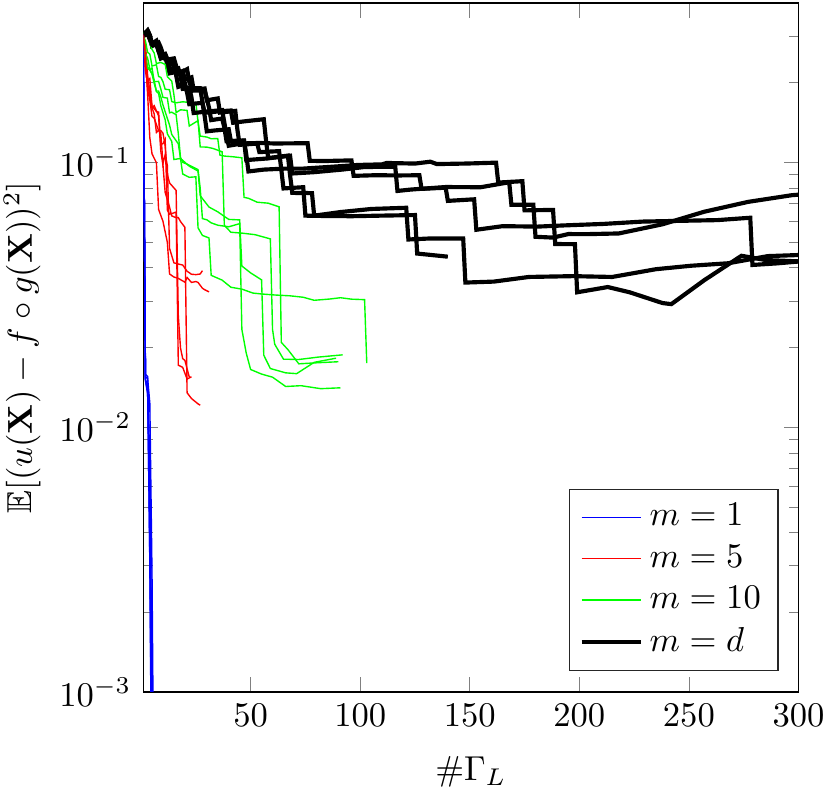} 
  \caption{Evolution of the mean squared error during the greedy Algorithm  \ref{alg:GreedyProfile}. The black curve $m=d$ is obtained by running Algorithm \ref{alg:GreedyProfile} with $g(x)=x$, the identity map.}
  \label{fig:CosOfNorm_OMP_f}
\end{subfigure}

  \caption{Isotropic function. 
  Performances of Algorithms \ref{alg:GreedyFeatureMaps} and \ref{alg:GreedyProfile} using $N=100$ samples (5 realizations).
  First, we construct $g$ using Algorithm \ref{alg:GreedyFeatureMaps} (left plot) and then, given $g$, we construct $f$ using Algorithm \ref{alg:GreedyProfile} (right plot).
  Both $J(g)$ and $\E[(u(\bX)-f\circ g(\bX))^2]$ are computed here on a large validation sample of size $2000$.}
  \label{fig:CosOfNorm_OMP}
\end{figure}

\setlength\tabcolsep{1 pt}
\begin{table}[t]
\small
\begin{tabularx}{1\textwidth}{|c|>{\centering\arraybackslash}X|c|c|}
\hline
 ~~~$m=1$~~~& $f(z) = \cos(\sqrt{z})$ & $g(\bx)=(x_1^2+\hdots+x_{20}^2)$ & $\#\Lambda_K=40$ \\\hline
 $m=2$& $f(z_1,z_2) = \cos(\sqrt{z_1^2+z_2})$ & 
 ~~~$g(\bx) = \left(\begin{array}{l} x_{1} \\ x_2^2+\hdots+x_{20}^2 \end{array} \right)$~~~
 & 
 ~~~$\#\Lambda_K=39$~~~
 \\\hline
 $\vdots$& $\vdots$ &  $\vdots$ & $\vdots$ \\\hline
 $m=19$& $f(z_1,\hdots,z_{19})=\cos(\sqrt{z_1^2+\hdots+z_{18}^2+z_{19}})$ &  $g(\bx) = \left(\begin{array}{l} x_{1}  \\ \vdots \\ x_{18} \\ x_{19}^2 + x_{20}^2  \end{array}\right)$ & $\#\Lambda_K=22$ \\\hline
 $m=20$& $f(z_1,\hdots,z_{20})=\cos(\sqrt{z_1^2+\hdots+z_{20}^2})$ & $g(\bx) = \left(\begin{array}{l} x_{1}  \\ \vdots \\ x_{20} \end{array}\right)$   & $\#\Lambda_K=20$ \\\hline
\end{tabularx}
\caption{Isotropic function. List of exact decompositions $u=f\circ g$ with polynomials $g\in\mathcal{G}_m^{\Lambda_K}$ with $\#\Lambda_K$ ranging from $40$ (and $m=1$) to $20$  (and $m=20$). This explains why, in Figure \ref{fig:CosOfNorm_OMP_g}, $J(g)$ drops to zero earlier in $\#\Lambda_K$ when $m$ is large.}
\label{tab:IsotropicFunctionExactDecompositions}
\end{table}

\subsection{Borehole function}\label{sec:Borehole}

Our second example is the commonly used Borehole function~\cite{borehole}, which models water flow through a borehole. It is a function of $d=8$ variables defined by
$$
 u(\bX) = \frac{2\pi T_u(H_u-H_l)}{\ln(r/r_w)\left(1+\frac{2LT_r}{\ln(r/r_w)r_w^2K_w} + \frac{T_r}{T_l}\right)},
$$
where $\bX$ is a random vector in $\R^d$ with independent components given by
$$
 \begin{array}{llll}
 X_1 = r_w &\sim \mathcal{N}(0.10, 0.0161812), \qquad\qquad\qquad&
 X_5 = r &\sim \log\mathcal{N}(7.71, 1.0056),  \\  
 X_2 = T_u &\sim \mathcal{U}[63070, 115600], &
 X_6 = H_u &\sim \mathcal{U}[990, 1110], \\
 X_3 = T_l &\sim \mathcal{U}[63.1, 116], &
 X_7 = H_l &\sim \mathcal{U}[700, 820], \\
 X_4 = L &\sim \mathcal{U}[1120, 1680], &
 X_8 = K_w &\sim \mathcal{U}[9855, 12045].
 \end{array}
$$

We first numerically illustrate Proposition \ref{prop:PoincareConditional}. Recall that this proposition states that, given $g\in\mathcal{G}_m$, there exists a function $f$ such that the mean squared error $\E[( u(\bX)-f(g(\bX)) \big)^2]$ is bounded by $J(g)$ multiplied by the Poincar\'e-type constant $\mathbb{C}(\bX|\mathcal{G}_m)$. In general, $\mathbb{C}(\bX|\mathcal{G}_m)$ is unknown.
We build three feature maps $g$: a linear map, a quadratic map, and a cubic map defined as the minimizers of $\widehat J(g)$ over the polynomial spaces
\begin{align*}
 \mathcal{G}_m^{\Lambda_\text{lin}} \quad\text{where}\quad
 \Lambda_\text{lin} &=\left\{\alpha\in\mathbb{N}^8: 1\leq\sum_{i=1}^8\alpha_i \leq 1\right\}, \quad\#\Lambda_\text{lin} = 8 ,\\
 \mathcal{G}_m^{\Lambda_\text{quad}} \quad\text{where}\quad
 \Lambda_\text{quad} &=\left\{\alpha\in\mathbb{N}^8: 1\leq\sum_{i=1}^8\alpha_i \leq 2\right\}, \quad\#\Lambda_\text{quad} = 44 ,\\
 \mathcal{G}_m^{\Lambda_\text{cub}} \quad\text{where}\quad
 \Lambda_\text{cub} &=\left\{\alpha\in\mathbb{N}^8: 1\leq\sum_{i=1}^8\alpha_i \leq 3\right\}, \quad\#\Lambda_\text{cub} = 164,
\end{align*}
respectively.
To compute these feature maps, we estimate $\widehat J(g)$ with $N=30$, $60$, or $150$ samples.
The dashed curves in Figure \ref{fig:Borehole_FixedG} are the resulting $J(g)$ (computed on a validation set of size $N=2000$) as a function of $m$. 
Once $g$ is built, we construct the profile $f$ using Algorithm \ref{alg:GreedyProfile} on the same sample.
The continuous lines in Figure \ref{fig:Borehole_FixedG} represent $\E[(u(\bX)-f\circ g(\bX))^2]$ (computed on the validation set).
As the sample size $N$ increases, we obtain a better profile function $f$, and the mean squared error decreases until it falls below $J(g)$. We also observe that the larger $m$ is, the higher $N$ must be to obtain a mean squared error below $J(g)$. Domination of the mean squared error by $J(g)$ is consistent with Proposition \ref{prop:PoincareConditional} with a Poincar\'e-type constant $\mathbb{C}(\bX|\mathcal{G}_m)$ that seems to be close to one for this benchmark.

In the limit $N\rightarrow\infty$, $g$ converges towards the optimal linear/quadratic/cubic feature map while the profile function $f$, built adaptively in Algorithm \ref{alg:GreedyProfile}, converges towards the solution of
$$ 
 \min_{f:\R^m\rightarrow\R}\E[(u(\bX)-f\circ g(\bX))^2].
$$
With a larger polynomial degree for $g$, the best achievable error $\min_{f:\R^m\rightarrow\R}\E[(u(\bX)-f\circ g(\bX))^2]$ is smaller and so we obtain a better approximation $f\circ g$ to $u$. Notice, however, that when the mean squared error is far above $J(g)$ (typically for large $m$), increasing the polynomial degree of $g$ does not significantly improve the approximation $f\circ g$. 
The interpretation is that if we cannot build a sufficiently accurate profile function $f$ (either because $m$ is too large or $N$ is too small), there is no benefit in having a complex (i.e., high polynomial degree) feature map $g$.

\begin{figure}[t]
  \centering 
  \begin{subfigure}{.31\textwidth}
  \centering
  \includegraphics[width=1\linewidth]{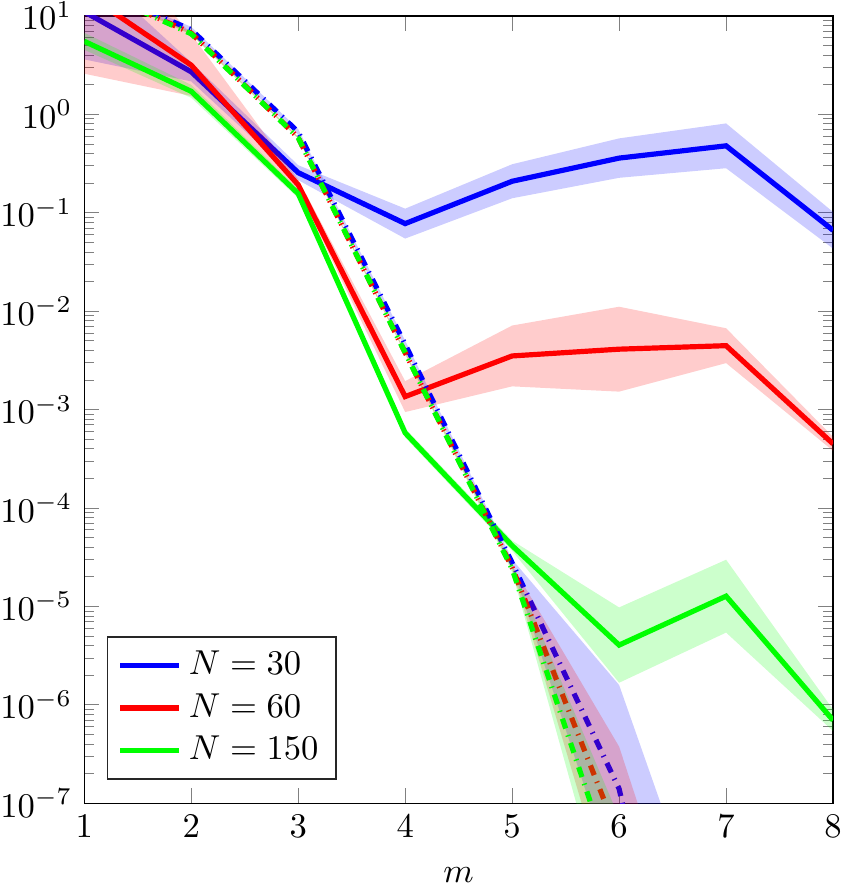} 
  \caption{Linear feature map}
  \label{fig:Borehole_FixedG_degree1}
\end{subfigure}
\begin{subfigure}{.31\textwidth}
  \centering
  \includegraphics[width=1\linewidth]{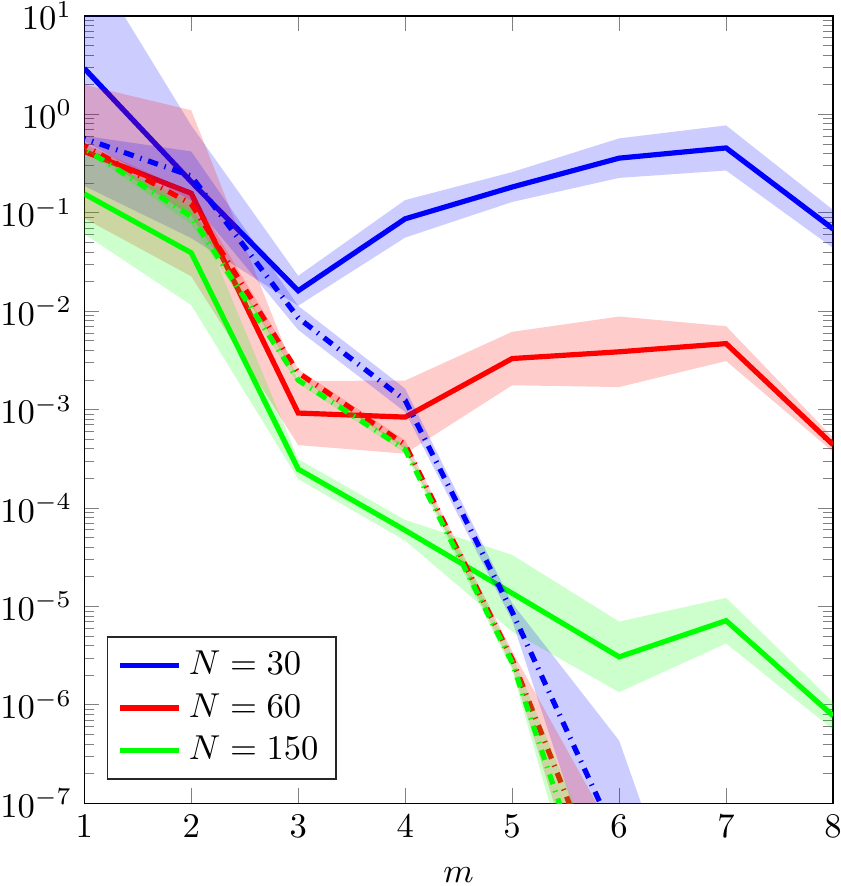} 
  \caption{Quadratic feature map}
  \label{fig:Borehole_FixedG_degree2}
\end{subfigure}
\begin{subfigure}{.31\textwidth}
  \centering
  \includegraphics[width=1\linewidth]{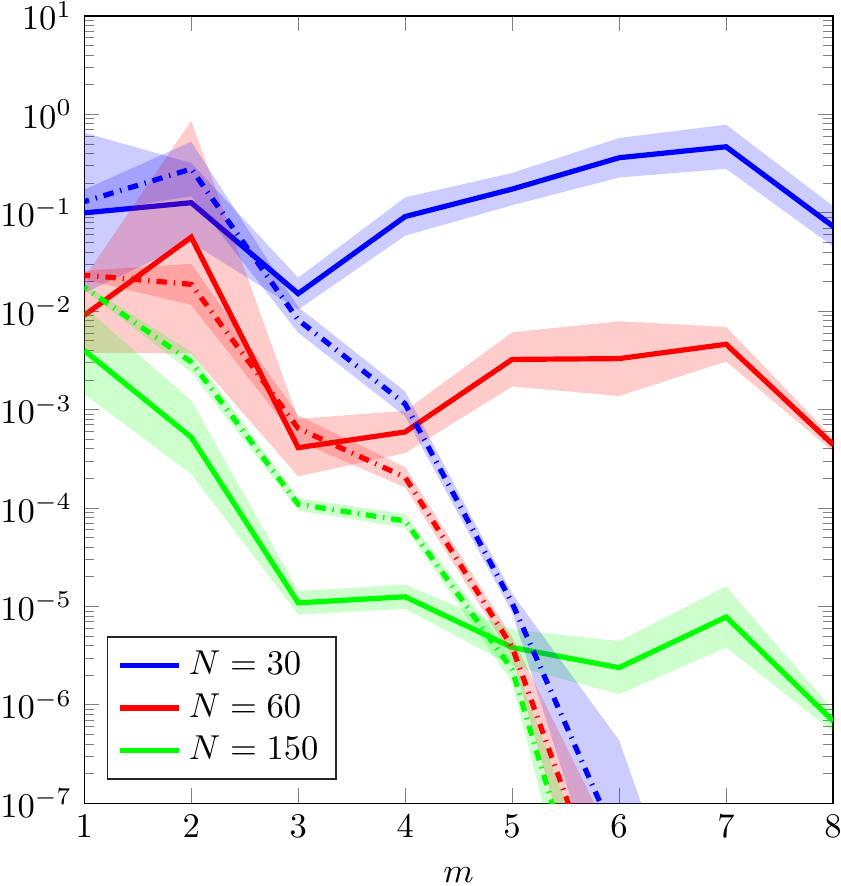} 
  \caption{Cubic feature map}
  \label{fig:Borehole_FixedG_degree3}
\end{subfigure}

  \caption{Borehole. Continuous lines: mean squared error $\E[(u(\bX)-f\circ g(\bX))^2]$, Dashed lines: cost function $J(g)$. The width of the shaded region corresponds to the standard deviation over $20$ experiments.
  The feature map $g$ is built by minimizing $\widehat J(g)$ using Algorithm \ref{alg:QuasiNewton} on samples of size $N\in\{30,60,150\}$.
  To build $f$, we employ Algorithm \ref{alg:GreedyProfile} on the same sample with bulk-chasing parameter $\theta=0.3$ and a five-fold cross-validation procedure to stop the iterations.
  }
  \label{fig:Borehole_FixedG}
\end{figure}

We now build both $g$ and $f$ adaptively using Algorithm \ref{alg:CV} with parameters $\theta=0.3$ and $\nu=5$ (from now on we use these parameters by default). Compared to the previous experiments where the polynomial degree of $g$ was fixed, the mean squared errors shown in Figure \ref{fig:Borehole_AdaptiveGandF} go to zero when $N\rightarrow\infty$, even for small $m$. 
Figure \ref{fig:Borehole_AdaptiveGandF_Cardinal} shows the cardinalities of $\Lambda_K$ and $\Gamma_L$ as functions of the intermediate dimension $m$. We clearly see that, for small $m$, our adaptive algorithm builds complex feature maps and simple profile functions. For large $m$, it is the other way around. 

From Figure \ref{fig:Borehole_AdaptiveGandF}, it seems that the optimal intermediate dimension $m$ depends on $N$: for small sample size $N=30$ or $N=60$, the best intermediate dimension is $m=2$ or $m=3$. For $N=150$, however, one clearly obtains better results with $m=d$, meaning without dimension reduction, i.e., $u(x)\approx f(x)$ with $g(x)=x$.

\begin{figure}[t]
  \centering 
\begin{subfigure}[t]{.45\textwidth}
  \centering
  \includegraphics[width=0.8\linewidth]{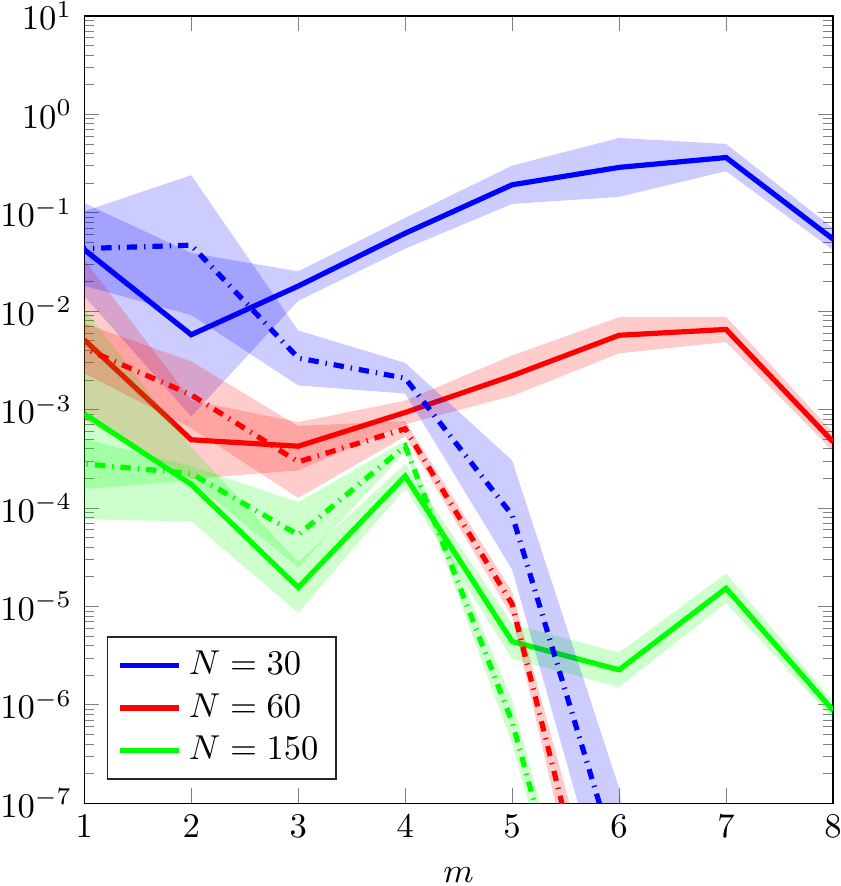} 
  \caption{Gradient-enhanced construction of $f$}
  \label{fig:Borehole_AdaptiveGandF}
\end{subfigure}
\begin{subfigure}[t]{.45\textwidth}
  \centering
  % include second image
  \includegraphics[width=0.75\linewidth]{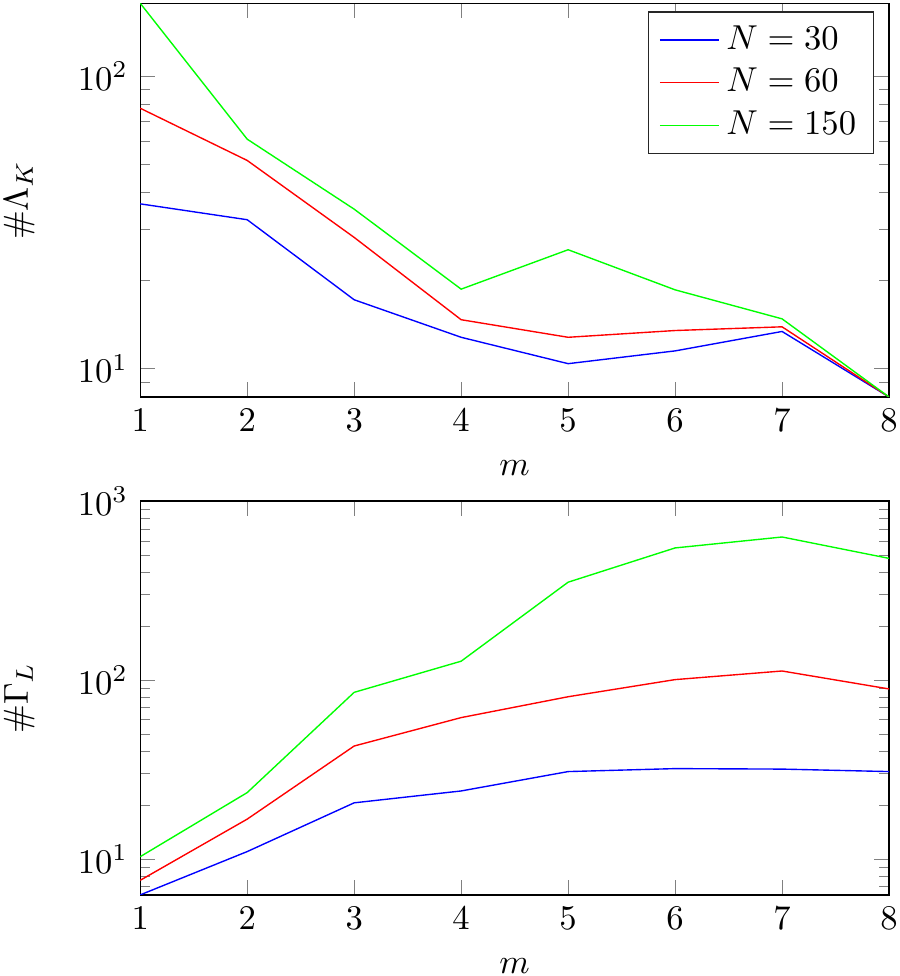} 
  \caption{Mean cardinality of $\Lambda_K$ (top) and of $\Gamma_L$ (bottom)}
  \label{fig:Borehole_AdaptiveGandF_Cardinal}
\end{subfigure}

  \caption{Borehole. Same settings as for Figure \ref{fig:Borehole_FixedG} but with a feature map $g$ built using the adaptive Algorithm \ref{alg:GreedyFeatureMaps}. The plots on the right show the complexity of $g$ and $f$ through the cardinalities of $\Lambda_K$ and $\Gamma_L$, respectively (mean over 20 experiments).
  }
  \label{fig:Borehole_AdaptiveG}
\end{figure}

\subsection{Composed function}\label{sec:ComposedFunction}

We consider now the benchmark introduced in \cite{grelier2018learning} defined as a deep composition of functions. We consider the function $u$ of $d=16$ variables defined by
\begin{align*}
 u(x) = h\Big(&h\big(h(h(x_1,x_2),h(x_3,x_4)),h(h(x_5,x_6),h(x_7,x_8))\big), \\ 
 &h\big(h(h(x_9,x_{10}),h(x_{11},x_{12})),h(h(x_{13},x_{14}),h(x_{15},x_{16}))\big)\Big),
\end{align*}
where $h(s,t) = 9^{-1}(1+st)^2$ and we let $\bX$ be the random vector with uniform measure on $[-1,1]^{16}$.
This function $u$ is a polynomial (as a composition of polynomials) and can readily be written as $u=f\circ g$ for $m=2,4,8$ with polynomials $f$ and $g$.

Numerical results are reported in Figure \ref{fig:DeepComposed}. For each choice of $N$ and $m$, after constructing the feature map $g$ via Algorithm~\ref{alg:GreedyFeatureMaps} and the cross-validation procedure in the first half of Algorithm~\ref{alg:CV}, we illustrate the benefits of the gradient-enhanced construction of the profile function $f$ by building it either with gradient-free least squares (i.e., by minimizing $\widehat{\mathcal{E}}_g (f) = \frac{1}{N}\sum_{i=1}^N (u({\bx}^{(i)})-f\circ g ({\bx}^{(i)}))^2$) or with gradient-enhanced least squares (i.e., by minimizing $\widehat{\mathcal{E}}_g^\nabla (f)$ in \eqref{eq:RMShat_gradientEnhanced}). 
For large $m$, the gradient-enhanced approach clearly outperforms the gradient-free approach, but for small $m$, both approaches perform equally. It seems that, for small $m$, the profile can be estimated accurately using evaluations of $u({\bx}^{(i)})$ only.
Since gradients are needed to construct $g$ regardless, our recommendation is always to use the gradient-enhanced approach to construct $f$, as it makes better use of the available information.

For this benchmark, it seems that $m=2$ is the best intermediate dimension for the considered range of sample sizes $N$. With this choice, the mean squared error can be reduced by around a factor of 10 over a full-dimensional function approximation scheme that simply uses $g = \text{Id}$ with the same sample.

\begin{figure}[t]
  \centering 
  \begin{subfigure}[t]{.31\textwidth}
  \centering
  \includegraphics[width=1\linewidth]{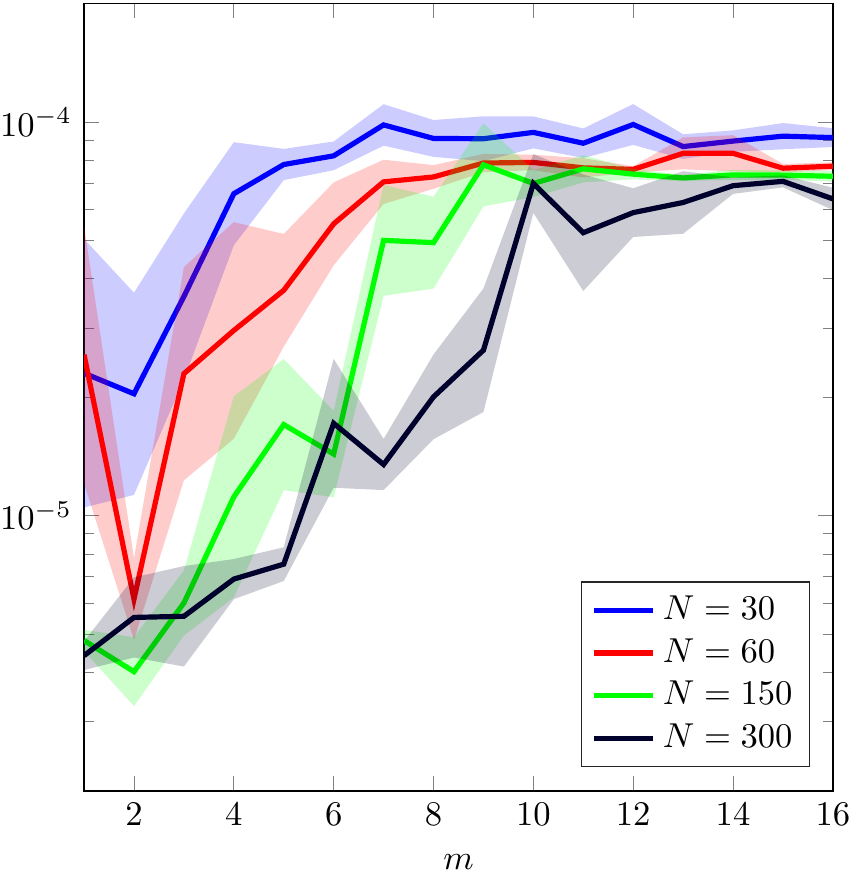} 
  \caption{Gradient-free construction of $f$}
  \label{fig:DeepComposed_AdaptiveGandF_GradFreeRegression}
  \end{subfigure}
\begin{subfigure}[t]{.31\textwidth}
  \centering
  \includegraphics[width=1\linewidth]{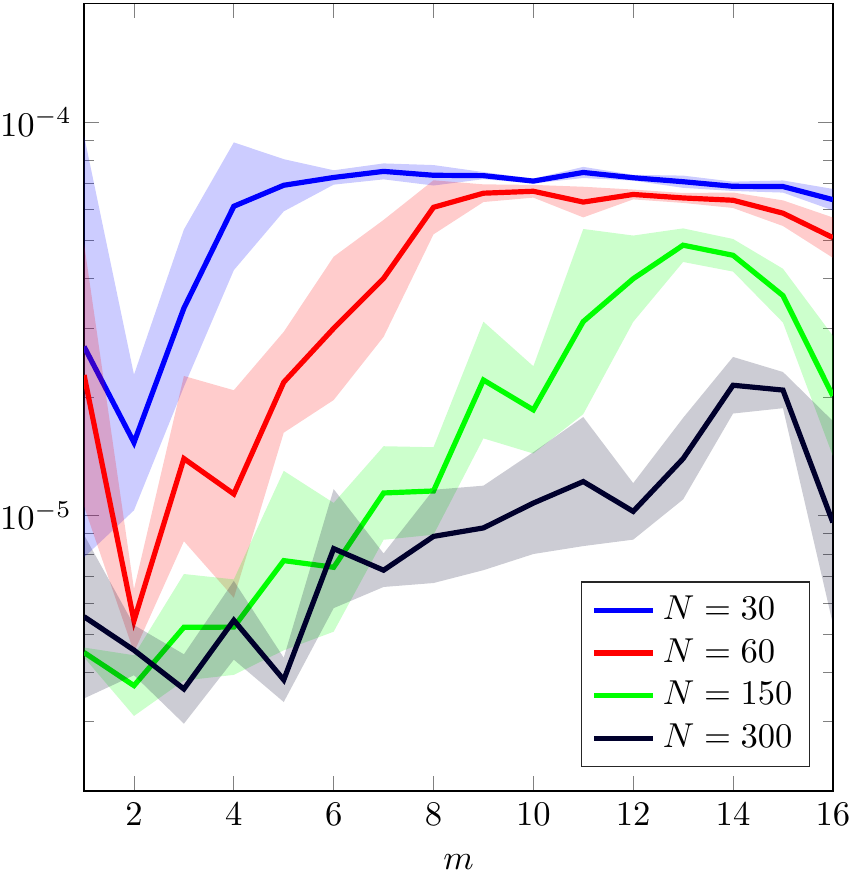} 
  \caption{Gradient-enhanced construction of $f$}
  \label{fig:DeepComposed_AdaptiveGandF}
\end{subfigure}
\begin{subfigure}[t]{.31\textwidth}
  \centering
  \includegraphics[width=1\linewidth]{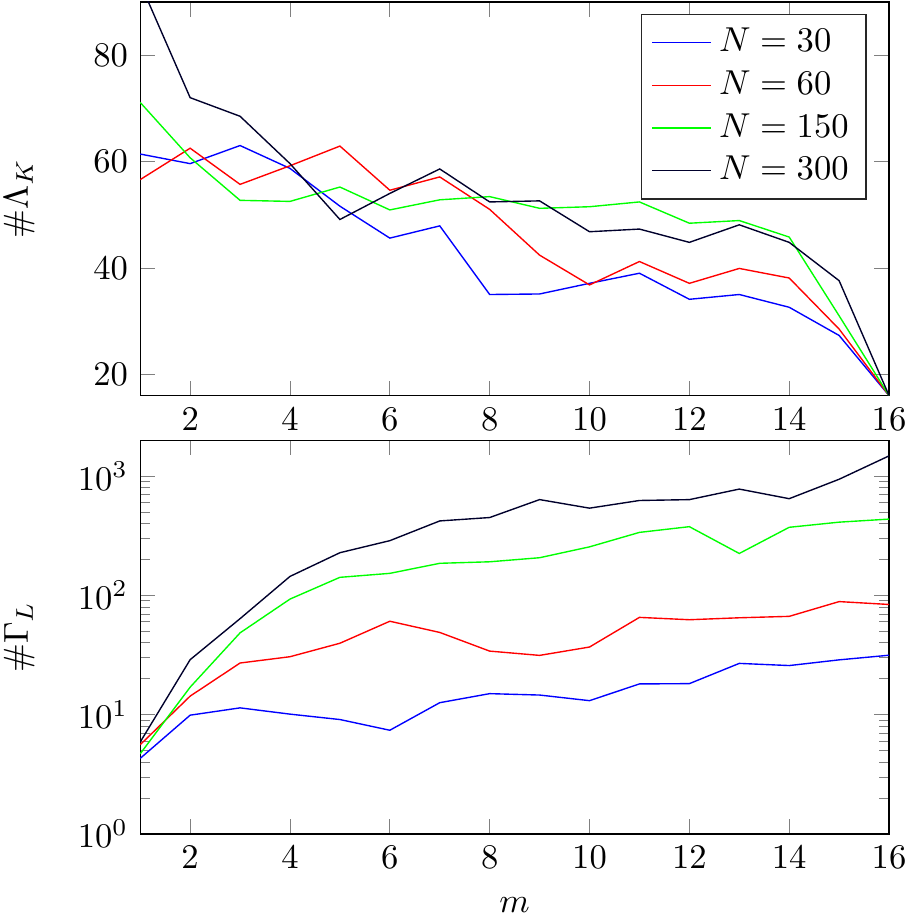} 
  \caption{Mean cardinality of $\Lambda_K$ (top) and of $\Gamma_L$ (bottom).}
  \label{fig:DeepComposed_AdaptiveGandF_Cardinal}
\end{subfigure}

  \caption{Composed function. 
  Mean squared error $\E[(u(\bX)-f\circ g(\bX))^2]$ (computed on a validation set of size $1000$) where $g$ and $f$ obtained by Algorithm \ref{alg:CV} ($\theta=0.3$ and $\nu=5$). The line (resp. the width of the shades) corresponds to the mean (resp. the variance) over $20$ experiments. Figure \ref{fig:DeepComposed_AdaptiveGandF_GradFreeRegression}: $f$ is built by minimizing the gradient-free mean square $\widehat{\mathcal{E}}_g (f) = \frac{1}{N}\sum_{i=1}^N (u({\bx}^{(i)})-f\circ g ({\bx}^{(i)}))^2$. Figure \ref{fig:DeepComposed_AdaptiveGandF}: $f$ is built by minimizing by minimizing $\widehat{\mathcal{E}}_g^\nabla (f)$, see \eqref{eq:RMShat_gradientEnhanced}.
  Figure \ref{fig:DeepComposed_AdaptiveGandF_Cardinal}: cardinalities of $\Lambda_K$ and of $\Gamma_L$ (with the gradient-enhanced construction of $f$).
  }
  \label{fig:DeepComposed}
\end{figure}

\subsection{Resonance frequency of a bridge}

Our last numerical experiment is a PDE-based model where the quantity of interest $u(\bx)$ is the smallest resonance frequency of a 2D structure which has the shape of a bridge, as shown in Figure \ref{fig:Morandi}. Here, $\bx$ parameterizes the Young modulus field of the structure. An important feature of this problem is that, while it relies on a complex numerical model, one can evaluate the gradient $\nabla u(\bx)$ with the same computational cost as that of an evaluation of $u(\bx)$, as we shall explain below.

To model the structure, we consider a linear elasticity problem in two spatial dimensions under plane stress assumption. 
After finite element discretization, the smallest resonance frequency $u(\bx)$ is defined as the minimum of a Rayleigh quotient
$$
 u(\bx) = \min_{v\in\R^n} \frac{v^T K(\bx) v}{v^T M v},
$$
where $K(\bx)\in\R^{n\times n}$ and $M\in\R^{n\times n}$ are the stiffness and the mass matrices given by
\begin{align*}
 K_{ij}(\bx) &= \int_\Omega \left\langle \frac{E(\bx)}{1+\nu}\varepsilon(\phi_i) +\frac{\nu E(\bx)}{1-\nu^2}\trace(\varepsilon(\phi_i)) I_2,\varepsilon(\phi_j)\right\rangle_\text{F}\d\Omega ,\\
 M_{ij} &= \int_\Omega \langle \phi_i,\phi_j \rangle\, \d\Omega.
\end{align*}
Here, $n=960$ is the number of nodes in the finite element mesh, $\phi_i:\Omega\rightarrow\R^2$ is the $i$-th finite element function, $\varepsilon(v)=\frac{1}{2}(\nabla v + \nabla v^T)\in\R^{2\times 2}$ is the strain tensor, $\langle\cdot,\cdot\rangle_F$ is the Frobenius scalar product in $\R^{2\times 2}$, and $\langle\cdot,\cdot\rangle$ the canonical scalar product in $\R^2$.
The Poisson coefficient is set to $\nu=0.3$ and the Young modulus field $E(\bx):\Omega\rightarrow\R$ is parameterized by a $d=32$-dimensional parameter $\bx\in\R^d$ as follows,
$$
 E(\bx) = \exp\left(\sum_{i=1}^{32} x_i \sqrt{\sigma_i}\psi_i \right),
$$
where $\psi_i:\Omega\rightarrow\R$ and $\sigma_i$ are the $i$-th leading eigenfunctions and eigenvalues of the Gaussian kernel $c(s,t)=\sqrt{5}\exp(-\|s-t\|_2^2/{20})$. We endow the parameter $\bX$ with the standard normal distribution on $\R^{32}$.

\begin{figure}[t]
  \centering 
  \begin{subfigure}[t]{.45\textwidth}
  \centering
  \includegraphics[width=1\linewidth]{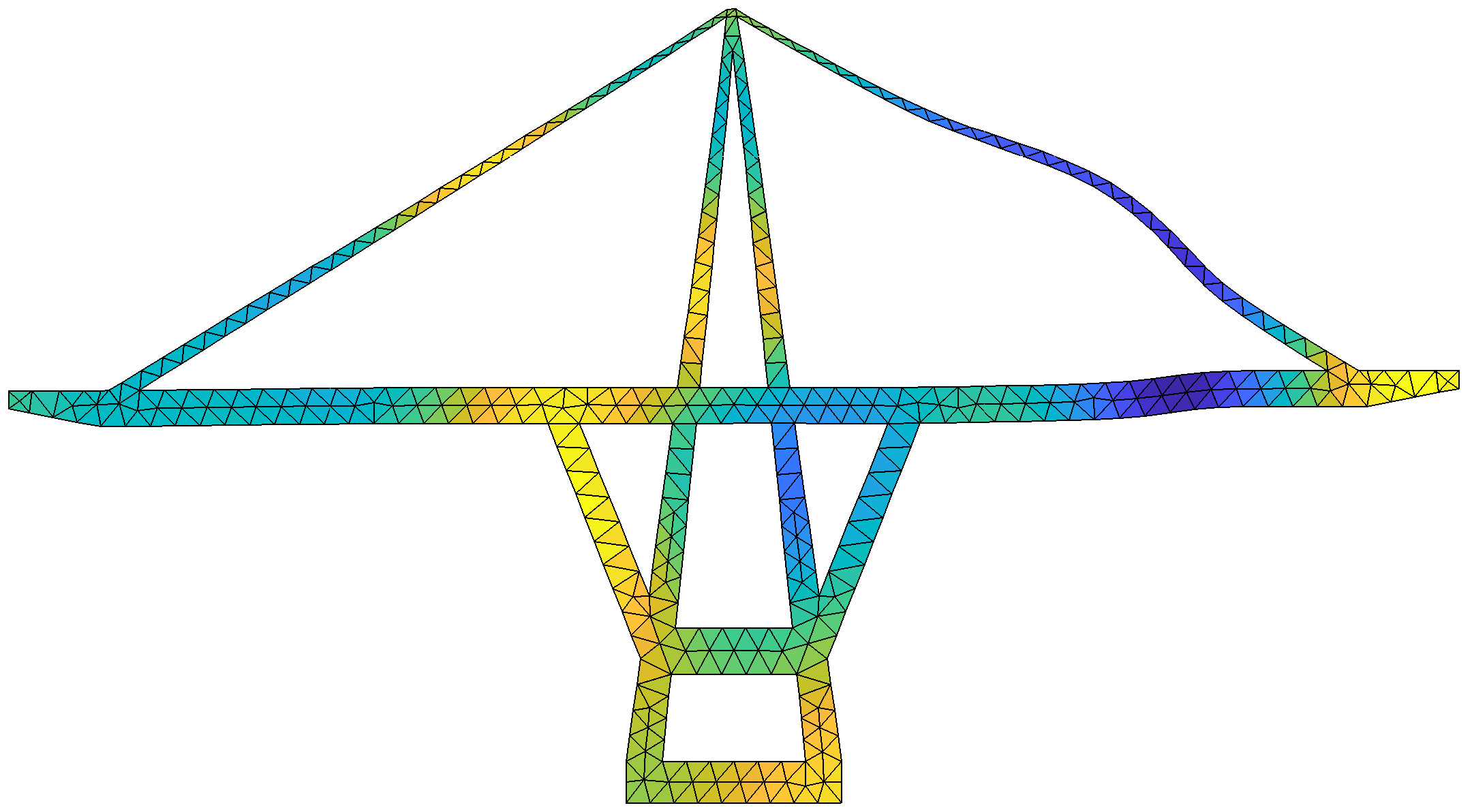} 
\end{subfigure}
\begin{subfigure}[t]{.45\textwidth}
  \centering
  \includegraphics[width=1\linewidth]{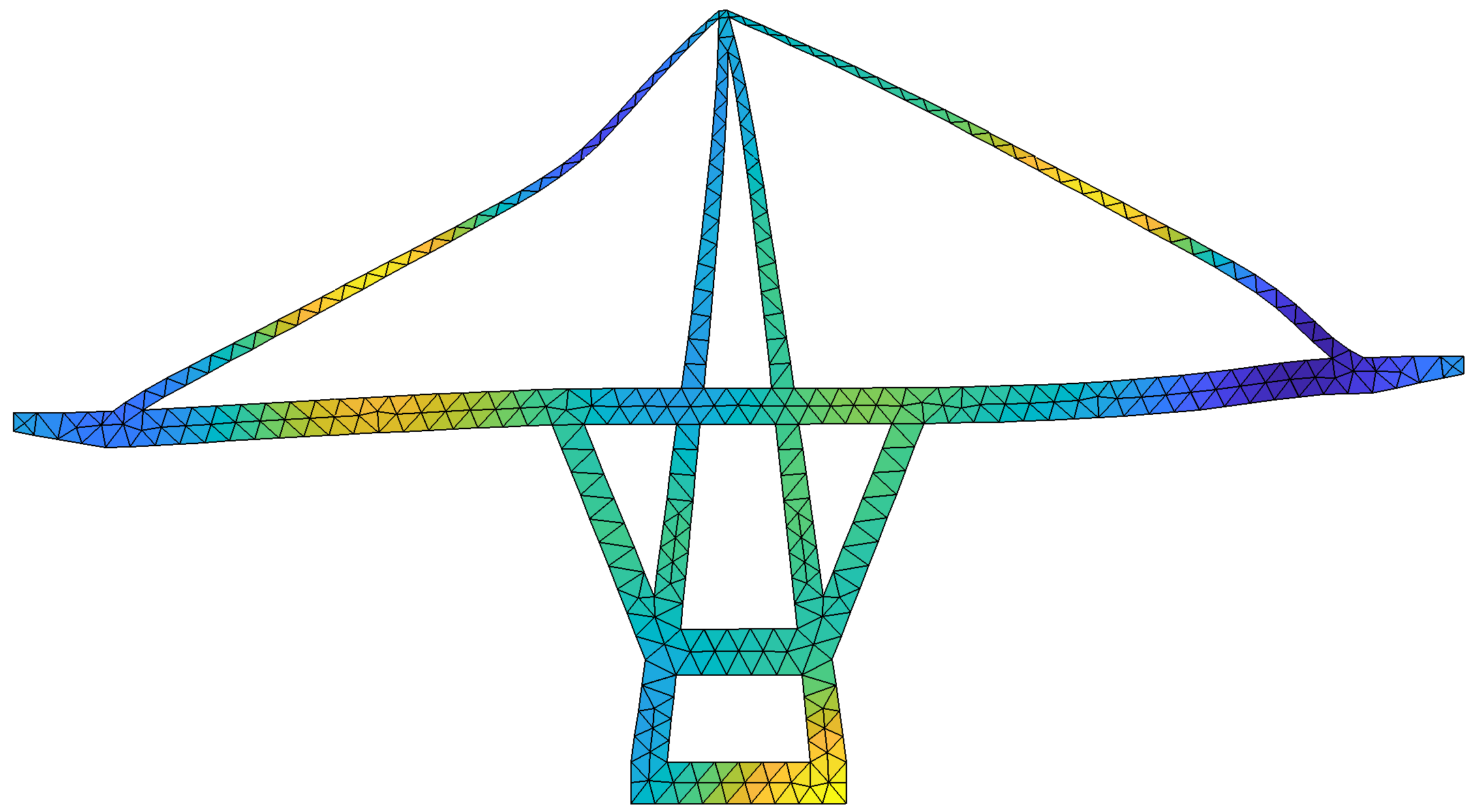} 
\end{subfigure}
\begin{subfigure}[t]{.45\textwidth}
  \centering
  \includegraphics[width=1\linewidth]{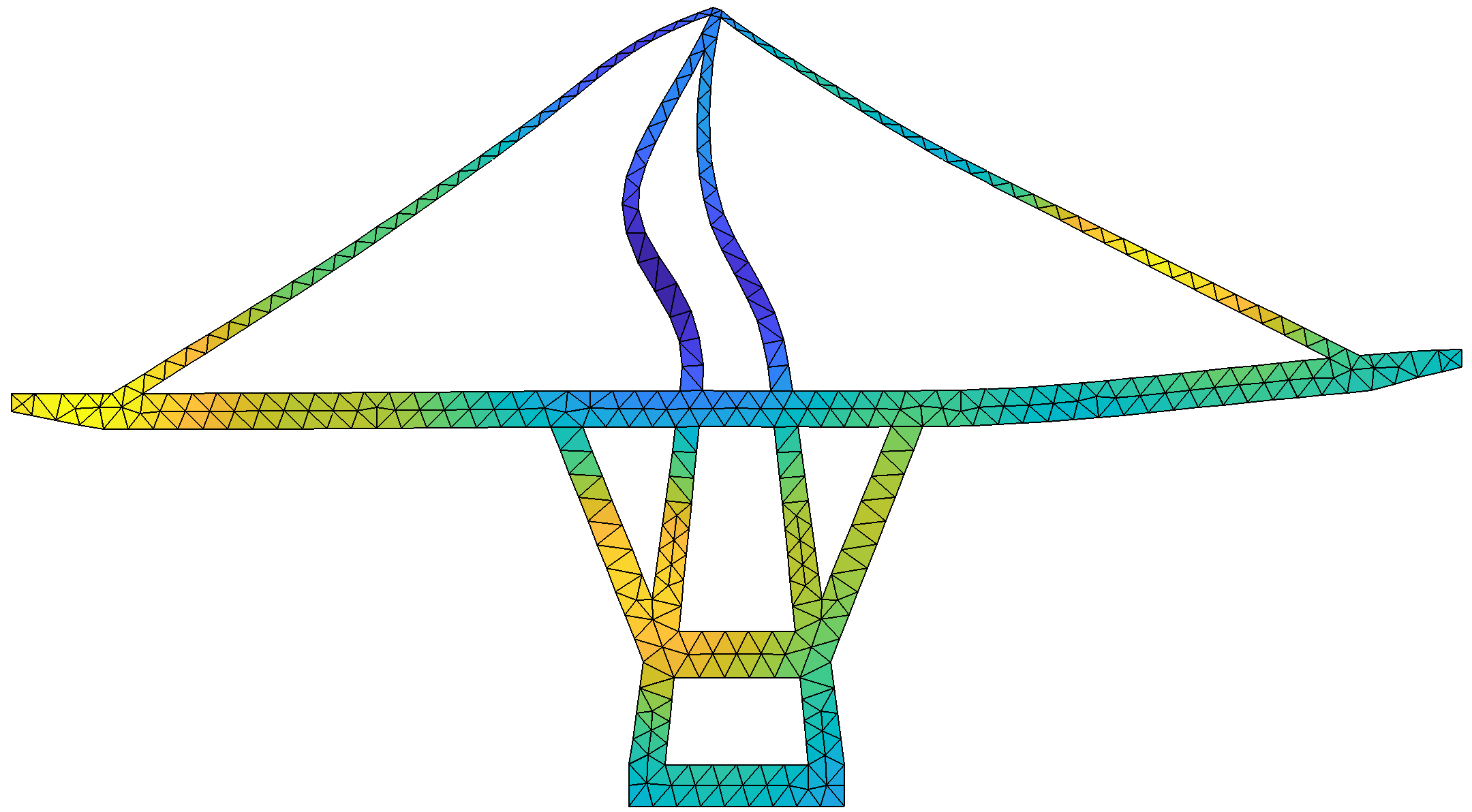} 
\end{subfigure}
\begin{subfigure}[t]{.45\textwidth}
  \centering
  \includegraphics[width=1\linewidth]{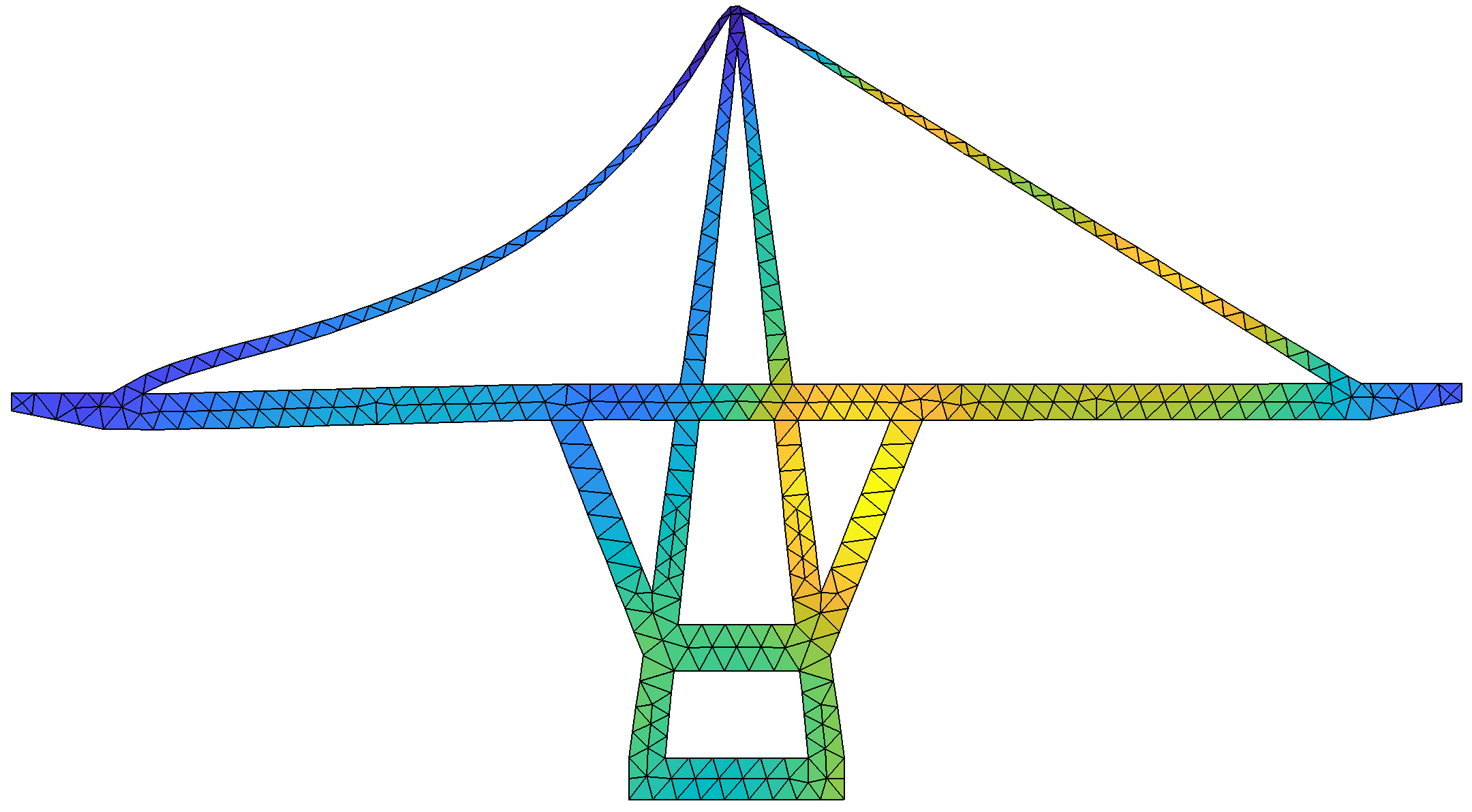} 
\end{subfigure}
  \caption{Resonance frequency of a bridge. Four realizations of the Young modulus field $E(\bX)$ (color of the elements) and the associated resonance mode $v(\bX)$ (displacement of the mesh).}
  \label{fig:Morandi}
\end{figure}

We denote by 
$$
 v(\bx) = \underset{v\in\R^n}{\text{argmin }} \frac{v^T K(\bx) v}{v^T M v}, 
$$
the minimizer of the Rayleigh quotient (i.e., the eigenvector associated to the eigenvalue/frequency $u(\bx)$). The $i$-th component of $\nabla u(\bx) = (\partial_{x_1} u(\bx) ,\cdots, \partial_{x_d} u(\bx))$ can be written as
\begin{equation}\label{eq:grad_Morandi}
 \partial_{x_i} u(\bx) = \frac{v(\bx)^T \big(\partial_{x_i} K(\bx) \big) v(\bx)}{v(\bx)^T M v(\bx)}.
\end{equation}
To show this, let us write $u(\bx) = R(v(\bx),\bx)$ where $R(v,x)=\frac{v^T K(\bx) v}{v^T M v}$ is the Rayleigh quotient. By definition of $v(\bx)$ we have $\nabla_v R(v(\bx),\bx) = 0$ so that a chain rule derivative yields $\partial_{x_i} u(\bx) = \nabla_v R(v(\bx),\bx)^T \partial_{x_i} v(\bx)+ \partial_{x_i} R(v(\bx),\bx) = \partial_{x_i} R(v(\bx),\bx)$, which is \eqref{eq:grad_Morandi}.
By definition of $E(\bx)$ and $K(\bx)$, the matrix $\partial_{x_i} K(\bx)$ is given by
$$
 \partial_{x_i} K_{kl}(\bx) = \int_\Omega \sqrt{\sigma_i}\psi_i \left\langle \frac{E(\bx)}{1+\nu}\varepsilon(\phi_k) +\frac{\nu E(\bx)}{1-\nu^2}\trace(\varepsilon(\phi_k)) I_2,\varepsilon(\phi_l)\right\rangle_\text{F}\d\Omega.
$$
The cost of assembling $\partial_{x_i} K$ for $1\leq i \leq d$ is negligible compared to  the cost of computing the eigenmode $v(\bx)$, which requires an expensive inverse power iteration method. In other words, once $v(\bx)$ is computed, one can evaluate both $u(\bx)$ and $\nabla u(\bx)$ almost for free.

In Table \ref{tab:Bridge} we report the performance of Algorithm \ref{alg:CV} on this benchmark, for a sample size $N=100$ and a range of values of $m$. The best performance is obtained with an intermediate dimension of $m=3$. For $m=8$ or $m=16$, the mean squared error is slightly higher than for $m=d$, meaning when we don't reduce the dimension. 
As before, we observe that a small intermediate dimension $m$ yields complex feature maps $g$ (i.e., large $\#\Lambda_K$) and simple profiles $f$ (i.e., small $\#\Gamma_L$). 

\begin{table}[t]
\footnotesize
\centering

\setlength\tabcolsep{0.1 pt}
\begin{tabularx}{1\textwidth}
{|>{\centering\arraybackslash}X|>{\centering\arraybackslash}X|>{\centering\arraybackslash}X|>{\centering\arraybackslash}X|>{\centering\arraybackslash}X|>{\centering\arraybackslash}X|>{\centering\arraybackslash}X|>{\centering\arraybackslash}X|>{\centering\arraybackslash}X|}
\hline
    & $m=1$        & $m=2$        & $m=3$        & $m=4$        & $m=6$        & $m=8$        & $m=16$       & $m=d=32$ \\\hline
  Mean$ \times 10^{12}$  & $1.6$ & $1.5$ & $\bf 1.1$ & $1.2$ & $1.3$ & $1.5$ & $1.6$ & $1.4$ \\\hline
  Std$ \times 10^{12}$  & $0.80$ & $0.69$ & $\bf 0.22$ & $0.24$ & $0.28$ & $0.83$ & $0.39$ & $0.43$  \\\hline
  $\#\Lambda_K$ &$148\,(\pm 64)$ & $129\,(\pm 45)$ & $91\,(\pm 21)$ & $80\,(\pm 23)$ & $64\,(\pm 16)$ & $57\,(\pm 9)$ & $51\,(\pm 1)$ & $32\,(\pm 0)$ \\\hline
  $\#\Gamma_L$ &$5\,(\pm 1)$ & $8\,(\pm 1)$ & $11\,(\pm 1)$ & $15\,(\pm 3)$ & $24\,(\pm 7)$ & $44\,(\pm 24)$ & $133\,(\pm 102)$ & $102\,(\pm 70)$ \\\hline
\end{tabularx}
\caption{Bridge. Mean and standard deviation (std) of the mean squared error $\E[(u(\bX)-f\circ g(\bX))^2]$ over $20$ experiments, where $g$ and $f$ are constructed using Algorithm \ref{alg:CV} with $N=100$ samples. The error $\E[(u(\bX)-f\circ g(\bX))^2]$ is computed on a (fixed) validation set of size $1000$. The last two lines of the table give the mean($\pm\,$std) of the cardinalities $\#\Lambda_K$ and $\#\Gamma_L$, which represent the complexity of $g$ and $f$, respectively.}
\label{tab:Bridge}
\end{table}

\section{Conclusion}

We have proposed and analyzed a novel framework for the dimension reduction of multivariate functions. Our approach relies on gradient evaluations of the model $u:\R^d\rightarrow\R$ and is a two-step procedure. First, we build a feature map $g:\R^d\rightarrow\R^m$ in a function space $\mathcal{G}_m$ by aligning the Jacobian of $g$ with the gradients of $u$. Second, we build a profile function $f:\R^m\rightarrow\R$ by minimizing the mean squared error between $u$ and $f\circ g$. 
We prove that having a finite Poincar\'e constant $\mathbb{C}(\bX|\mathcal{G}_m)$ ensures good theoretical properties of the feature map---namely that the objective used to identify $g$ bounds the $L^2$ error between $u$ and its approximation. The Poincar\'e constant depends both on the probability measure of the inputs $\bX$ and on the feature space $\mathcal{G}_m$. In practice we observe good approximation performance using polynomial spaces $\mathcal{G}_m$, constructed via a greedy adaptive procedure, but we cannot easily check that $\mathbb{C}(\bX|\mathcal{G}_m)<\infty$ for this case. Indeed, theoretically guaranteeing that $\mathbb{C}(\bX|\mathcal{G}_m)< \infty$ for a computationally feasible space of nonlinear feature maps $\mathcal{G}_m$ remains a challenge.

Our numerical experiments also illustrate the role of the intermediate dimension $m$ in this setting. It is natural to ask what is the \emph{intrinsic} intermediate dimension $m$ of a model $u$? From a theoretical perspective, we argue that this question is void without specifying a function class $\mathcal{G}_m$ for $g$. For instance, we can talk about the \emph{linear} or \emph{quadratic} intrinsic intermediate dimension of $u$ as the smallest $m$ such that there exists a linear or a quadratic $g$ so that the error $\E[(u(\bX)-f\circ g(\bX))^2]$ is less than a prescribed tolerance for some $f:\R^m\rightarrow\R$. The OMP-type algorithm we propose, which adapts the complexity of $\mathcal{G}_m$ to the sample size, then makes the interpretation of $m$ more complicated.

A useful alternative question is how to optimally select the intermediate dimension $m$ in practice? For now, we have no way to select it \textit{a priori}. In our numerical tests, we run the algorithm for all possible values of $m=1,\hdots,d$ and select the intermediate dimension which yields the lowest cross-validation error. We have observed that the intermediate dimension which yields the smallest reconstruction error depends on the sample size $N$: for instance, in the small sample size regime, an intermediate dimension of $m=2$ or $3$ might yield better approximation while, in the large sample size regime, no dimension reduction, i.e., $m=d$, could be a better choice. This trend depends very much on the target function $u$, and we show examples where an intermediate value of $m$ is best over a range of sample sizes. 

The minimization of the function $J(g)$ turns out to be quite a challenging task. While the quasi-Newton method proposed here is generally effective, 
recent work \cite{lasserre2020minimizing} may offer a novel optimization perspective to address the essential problem of minimizing sums of generalized Rayleigh quotients.

Another interesting direction motivated by the present work is the recursive construction of approximations of the form $f_{k}\circ f_{k-1}\circ \hdots \circ f_1$, where each $f_i$ is built using gradients of $u$. This composition is related to deep neural network architectures for function approximation, and may offer a perspective on the choice of latent space and internal dimension in such methods.

\section*{Acknowledgment}

The authors gratefully acknowledge support from the Inria associate team UNQUESTIONABLE.
CP and OZ also acknowledge support from CIROQUO consortium.
DB and YMM also acknowledge support from the US Department of Energy, Office of Advanced Scientific Computing Research, AEOLUS project.

\bibliographystyle{siam}

\begin{thebibliography}{10}

\bibitem{absil2009optimization}
{\sc P.-A. Absil, R.~Mahony, and R.~Sepulchre}, {\em Optimization algorithms on
  matrix manifolds}, Princeton University Press, 2009.

\bibitem{Adragni09}
{\sc K.~P. Adragni and R.~D. Cook}, {\em Sufficient dimension reduction and
  prediction in regression}, Philosophical Transactions of the Royal Society A:
  Mathematical, Physical and Engineering Sciences, 367 (2009), pp.~4385--4405.

\bibitem{nouy_anthony_2020_3653971}
{\sc N.~Anthony, G.~Erwan, and G.~Loic}, {\em Approximationtoolbox}, Feb. 2020.

\bibitem{bakry2008simple}
{\sc D.~Bakry, F.~Barthe, P.~Cattiaux, A.~Guillin, et~al.}, {\em A simple proof
  of the poincar{\'e} inequality for a large class of probability measures},
  Electronic Communications in Probability, 13 (2008), pp.~60--66.

\bibitem{beck2013sparsity}
{\sc A.~Beck and Y.~C. Eldar}, {\em Sparsity constrained nonlinear
  optimization: Optimality conditions and algorithms}, SIAM Journal on
  Optimization, 23 (2013), pp.~1480--1509.

\bibitem{boucheron2013concentration}
{\sc S.~Boucheron, G.~Lugosi, and P.~Massart}, {\em Concentration inequalities:
  A nonasymptotic theory of independence}, Oxford university press, 2013.

\bibitem{brennan2020greedy}
{\sc M.~C. Brennan, D.~Bigoni, O.~Zahm, A.~Spantini, and Y.~Marzouk}, {\em
  Greedy inference with structure-exploiting lazy maps}, arXiv preprint
  arXiv:1906.00031,  (2020).

\bibitem{chkifa2015breaking}
{\sc A.~Chkifa, A.~Cohen, and C.~Schwab}, {\em Breaking the curse of
  dimensionality in sparse polynomial approximation of parametric pdes},
  Journal de Math{\'e}matiques Pures et Appliqu{\'e}es, 103 (2015),
  pp.~400--428.

\bibitem{cohen2012capturing}
{\sc A.~Cohen, I.~Daubechies, R.~DeVore, G.~Kerkyacharian, and D.~Picard}, {\em
  Capturing ridge functions in high dimensions from point queries},
  Constructive Approximation, 35 (2012), pp.~225--243.

\bibitem{cohen2018multivariate}
{\sc A.~Cohen and G.~Migliorati}, {\em Multivariate approximation in downward
  closed polynomial spaces}, in Contemporary Computational Mathematics-A
  celebration of the 80th birthday of Ian Sloan, Springer, 2018, pp.~233--282.

\bibitem{constantine2015}
{\sc P.~G. Constantine}, {\em Active subspaces: Emerging ideas for dimension
  reduction in parameter studies}, SIAM, 2015.

\bibitem{Constantine2014a}
{\sc P.~G. Constantine, E.~Dow, and Q.~Wang}, {\em Active subspace methods in
  theory and practice: applications to kriging surfaces}, SIAM Journal on
  Scientific Computing, 36 (2014), pp.~A1500--A1524.

\bibitem{cookweis1991}
{\sc R.~D. Cook and S.~Weisberg}, {\em Discussion of sliced inverse regression
  for dimension reduction}, Journal of the American Statistical Association, 86
  (1991), pp.~328--332.

\bibitem{cui2020data}
{\sc T.~Cui and O.~Zahm}, {\em Data-free likelihood-informed dimension
  reduction of bayesian inverse problems},  (2020).

\bibitem{dennis1977quasi}
{\sc J.~E. Dennis, Jr and J.~J. Mor{\'e}}, {\em Quasi-newton methods,
  motivation and theory}, SIAM review, 19 (1977), pp.~46--89.

\bibitem{fornasier2012learning}
{\sc M.~Fornasier, K.~Schnass, and J.~Vybiral}, {\em Learning functions of few
  arbitrary linear parameters in high dimensions}, Foundations of Computational
  Mathematics, 12 (2012), pp.~229--262.

\bibitem{golub2013matrix}
{\sc G.~H. Golub and C.~F. Van~Loan}, {\em Matrix computations}, vol.~3, JHU
  press, 2013.

\bibitem{grelier2018learning}
{\sc E.~Grelier, A.~Nouy, and M.~Chevreuil}, {\em Learning with tree-based
  tensor formats}, arXiv preprint arXiv:1811.04455,  (2018).

\bibitem{griewank1989automatic}
{\sc A.~Griewank et~al.}, {\em On automatic differentiation}, Mathematical
  Programming: recent developments and applications, 6 (1989), pp.~83--107.

\bibitem{hokanson2018data}
{\sc J.~M. Hokanson and P.~G. Constantine}, {\em Data-driven polynomial ridge
  approximation using variable projection}, SIAM Journal on Scientific
  Computing, 40 (2018), pp.~A1566--A1589.

\bibitem{kokiopoulou2011trace}
{\sc E.~Kokiopoulou, J.~Chen, and Y.~Saad}, {\em Trace optimization and
  eigenproblems in dimension reduction methods}, Numerical Linear Algebra with
  Applications, 18 (2011), pp.~565--602.

\bibitem{krantz2012implicit}
{\sc S.~G. Krantz and H.~R. Parks}, {\em The implicit function theorem:
  history, theory, and applications}, Springer Science \& Business Media, 2012.

\bibitem{lam2020multifidelity}
{\sc R.~R. Lam, O.~Zahm, Y.~M. Marzouk, and K.~E. Willcox}, {\em Multifidelity
  dimension reduction via active subspaces}, SIAM Journal on Scientific
  Computing, 42 (2020), pp.~A929--A956.

\bibitem{lasserre2020minimizing}
{\sc J.~B. Lasserre, V.~Magron, S.~Marx, and O.~Zahm}, {\em Minimizing rational
  functions: a hierarchy of approximations via pushforward measures}, arXiv
  preprint arXiv:2012.05793,  (2020).

\bibitem{lataniotis2020extending}
{\sc C.~Lataniotis, S.~Marelli, and B.~Sudret}, {\em Extending classical
  surrogate modeling to high dimensions through supervised dimensionality
  reduction: a data-driven approach}, International Journal for Uncertainty
  Quantification, 10 (2020).

\bibitem{laurent2019overview}
{\sc L.~Laurent, R.~Le~Riche, B.~Soulier, and P.-A. Boucard}, {\em An overview
  of gradient-enhanced metamodels with applications}, Archives of Computational
  Methods in Engineering, 26 (2019), pp.~61--106.

\bibitem{lee2013general}
{\sc K.-Y. Lee, B.~Li, F.~Chiaromonte, et~al.}, {\em A general theory for
  nonlinear sufficient dimension reduction: Formulation and estimation}, Annals
  of Statistics, 41 (2013), pp.~221--249.

\bibitem{li2018sufficient}
{\sc B.~Li}, {\em Sufficient dimension reduction: Methods and applications with
  R}, CRC Press, 2018.

\bibitem{li1991sliced}
{\sc K.-C. Li}, {\em Sliced inverse regression for dimension reduction},
  Journal of the American Statistical Association, 86 (1991), pp.~316--327.

\bibitem{migliorati2015adaptive}
{\sc G.~Migliorati}, {\em Adaptive polynomial approximation by means of random
  discrete least squares}, in Numerical Mathematics and Advanced
  Applications-ENUMATH 2013, Springer, 2015, pp.~547--554.

\bibitem{migliorati2019adaptive}
\leavevmode\vrule height 2pt depth -1.6pt width 23pt, {\em Adaptive
  approximation by optimal weighted least-squares methods}, SIAM Journal on
  Numerical Analysis, 57 (2019), pp.~2217--2245.

\bibitem{parente2020generalized}
{\sc M.~T. Parente, J.~Wallin, B.~Wohlmuth, et~al.}, {\em Generalized bounds
  for active subspaces}, Electronic Journal of Statistics, 14 (2020),
  pp.~917--943.

\bibitem{peng2016polynomial}
{\sc J.~Peng, J.~Hampton, and A.~Doostan}, {\em On polynomial chaos expansion
  via gradient-enhanced $\ell$1-minimization}, Journal of Computational Physics,
  310 (2016), pp.~440--458.

\bibitem{pinkus2015ridge}
{\sc A.~Pinkus}, {\em Ridge functions}, vol.~205, Cambridge University Press,
  2015.

\bibitem{plessix2006review}
{\sc R.-E. Plessix}, {\em A review of the adjoint-state method for computing
  the gradient of a functional with geophysical applications}, Geophysical
  Journal International, 167 (2006), pp.~495--503.

\bibitem{saltelli2008global}
{\sc A.~Saltelli, M.~Ratto, T.~Andres, F.~Campolongo, J.~Cariboni, D.~Gatelli,
  M.~Saisana, and S.~Tarantola}, {\em Global sensitivity analysis: the primer},
  John Wiley \& Sons, 2008.

\bibitem{scheiblechner2007complexity}
{\sc P.~Scheiblechner}, {\em On the complexity of deciding connectedness and
  computing betti numbers of a complex algebraic variety}, Journal of
  Complexity, 23 (2007), pp.~359--379.

\bibitem{stewart1990matrix}
{\sc G.~W. Stewart}, {\em Matrix perturbation theory},  (1990).

\bibitem{borehole}
{\sc S.~Surjanovic and D.~Bingham}, {\em Virtual library of simulation
  experiments}, 2013.

\bibitem{tropp2007signal}
{\sc J.~A. Tropp and A.~C. Gilbert}, {\em Signal recovery from random
  measurements via orthogonal matching pursuit}, IEEE Transactions on
  information theory, 53 (2007), pp.~4655--4666.

\bibitem{villani2008optimal}
{\sc C.~Villani}, {\em Optimal transport: old and new}, vol.~338, Springer
  Science \& Business Media, 2008.

\bibitem{wang2018efficient}
{\sc X.~Wang, L.~Wang, and Y.~Xia}, {\em An efficient global optimization
  algorithm for maximizing the sum of two generalized rayleigh quotients},
  Computational and Applied Mathematics, 37 (2018), pp.~4412--4422.

\bibitem{wu2008kernel}
{\sc H.-M. Wu}, {\em Kernel sliced inverse regression with applications to
  classification}, Journal of Computational and Graphical Statistics, 17
  (2008), pp.~590--610.

\bibitem{yeh2008nonlinear}
{\sc Y.-R. Yeh, S.-Y. Huang, and Y.-J. Lee}, {\em Nonlinear dimension reduction
  with kernel sliced inverse regression}, IEEE transactions on Knowledge and
  Data Engineering, 21 (2008), pp.~1590--1603.

\bibitem{zahm2020gradient}
{\sc O.~Zahm, P.~G. Constantine, C.~Prieur, and Y.~M. Marzouk}, {\em
  Gradient-based dimension reduction of multivariate vector-valued functions},
  SIAM Journal on Scientific Computing, 42 (2020), pp.~A534--A558.

\bibitem{zahm2018certified}
{\sc O.~Zahm, T.~Cui, K.~Law, A.~Spantini, and Y.~Marzouk}, {\em Certified
  dimension reduction in nonlinear bayesian inverse problems}, arXiv preprint
  arXiv:1807.03712,  (2018).

\bibitem{zhang2019learning}
{\sc G.~Zhang, J.~Zhang, and J.~Hinkle}, {\em Learning nonlinear level sets for
  dimensionality reduction in function approximation}, in Advances in Neural
  Information Processing Systems, 2019, pp.~13199--13208.

\bibitem{zhang2013optimizing}
{\sc L.-H. Zhang}, {\em On optimizing the sum of the rayleigh quotient and the
  generalized rayleigh quotient on the unit sphere}, Computational Optimization
  and Applications, 54 (2013), pp.~111--139.

\bibitem{zhang2014self}
\leavevmode\vrule height 2pt depth -1.6pt width 23pt, {\em On a
  self-consistent-field-like iteration for maximizing the sum of the rayleigh
  quotients}, Journal of Computational and Applied Mathematics, 257 (2014),
  pp.~14--28.

\end{thebibliography}

\appendix

\section{Link with the loss function introduced in \cite{zhang2019learning}}\label{appA}

As in Example~\ref{example:C1diffeo}, let $\phi:\calX\rightarrow\calX$ be a $C^1$-diffeomorphism and let $g:\calX\rightarrow\R^m$ be a feature map defined by $g(\bx) = (\phi_1(\bx),\hdots,\phi_m(\bx))$. In \cite{zhang2019learning}, the diffeomorphism $\phi$ is built by minimizing the loss function
$$
 \mathcal{L}_\omega(\phi) \coloneqq \E\left[ \sum_{i=1}^d \omega_i \left\langle \frac{\nabla \phi_{i}(\bX)}{\|\nabla \phi_{i}(\bX)\|},  \nabla u({\bX})  \right\rangle^2 \right],
$$
where $\omega=(\omega_1,\hdots,\omega_d)\in\R^d_{\geq0}$ are non-negative weights which are arbitrarily chosen.
To link this loss function with the proposed cost function $J(g)$, let us assume that the orthogonality condition
\begin{equation}\label{eq:OrthogonalityCondition}
 \nabla\phi_i(\bx)^T \nabla\phi_j(\bx) = 0  ,
\end{equation}
holds for any $i\neq j$ and for any $\bx\in\calX$. Under this assumption, the cost function $J(g)$ can be written as
\begin{align*}
 J(g) &= \E\left[ \big\| (I_d - \Pi_{\mathrm{range}(\nabla g({\bX})^T)})  \nabla u({\bX}) \big\|_2^2 \right]\\
 &\overset{\eqref{eq:OrthogonalityCondition}}{=} \E\left[ \sum_{i=m+1}^d  \left\langle \frac{\nabla \phi_{i}(\bX)}{\|\nabla \phi_{i}(\bX)\|},  \nabla u({\bX})  \right\rangle^2 \right] \\
 &= \mathcal{L}_\omega(\phi),
\end{align*}
where the last equality is obtained by letting
$$
 \omega = (\underbrace{0,\hdots,0}_{m\text{ times}} , \underbrace{1,\hdots,1}_{d-m\text{ times}}).
$$

In \cite{zhang2019learning}, the loss function $\mathcal{L}_\omega(\phi)$ is used without ensuring the orthogonality condition \eqref{eq:OrthogonalityCondition} and no theoretical justification is provided. 
For instance, without condition \eqref{eq:OrthogonalityCondition}, it is unclear whether $\mathcal{L}_\omega(\phi)=0 $ implies $u(\bx)=f\circ g(\bx)$ or, more critically, if $u(\bx)=f\circ g(\bx)$ implies $\mathcal{L}_\omega(\phi)=0 $.

% \begin{example}
%  Let $u=f\circ g$ with $g(\bx)=A^T\bx$ for some matrix $A\in\R^{d\times m}$. Let $B\in\R^{d\times(d-m)}$ be a matrix such that $[A, B]\in\R^{d\times d}$ forms an invertible matrix and assume for simplicity that each column $B_1,\hdots,B_{d-m}$ have unit norm.
%  Then $\phi(\bx) = [A, B]\bx$ is a $C^1$-diffeomorphism.
%  We can write
%  \begin{align*}
%   \mathcal{L}_\omega(\phi) 
%   &= \E\left[ \sum_{i=1}^{d-m}  \left\langle B_i,  \nabla u({\bX})  \right\rangle^2 \right] \\
%   &= \E\left[  \| B^T A \nabla f(A^T\bX)  \|^2 \right] \\
% %   &= \E\left[  \trace \left( B^T A \nabla f(A^T\bX) \nabla f(A^T\bX) ^T A^T B \right)  \right] \\
% %   &=  \trace \left( B^T A \E\left[ \nabla f(A^T\bX) \nabla f(A^T\bX) ^T \right]  A^T B \right)  \\
%   &=  \trace \left( B^T A M  A^T B \right)  \\
%   &\geq \lambda_{\min}(M) \left\| B^T A \right\|_F^2,
%  \end{align*} 
%  where $M = \E\left[ \nabla f(A^T\bX) \nabla f(A^T\bX) ^T \right]$ and where $\|\cdot\|_F$  denotes the Frobenius norm.
%  Then, assuming $\lambda_{\min}(M)>0$, we have
% \end{example}

\section{Proof of Proposition \ref{prop:gradR}}\label{proof:gradR}

 We use the notation $M_\mathrm{sym}=(M+M^T)/2$ for the symmetric part of a square matrix $M$.
 For any $\|\delta G\|\leq \varepsilon$ we can write
 \begin{align*}
  (G+\delta G)^T A(\bX)(G+\delta G) 
  &= G^T A(\bX)G + 2(\delta G^T A(\bX) G )_\mathrm{sym} +\mathcal{O}(\|\delta G\|^2) ,
 \end{align*}
 and
 \begin{align*}
  \big( (G+ &\delta G)^T B(\bX) (G+\delta G)\big)^{-1}  \\
  &= \left( G^T B(\bX)G + 2(\delta G^T B(\bX) G)_\mathrm{sym} +\mathcal{O}(\|\delta G\|^2) \right)^{-1}  \\
  &= (G^T B(\bX)G )^{-1} 
  - 2(G^T B(\bX)G )^{-1}  (\delta G^T B(\bX) G)_\mathrm{sym} (G^T B(\bX)G )^{-1} 
  + \mathcal{O}(\|\delta G\|^2).
 \end{align*}
 Multiplying the two above quantities yields
 \begin{align*}
  \Big((G+& \delta G)^T A(\bX)(G+\delta G) \Big)\Big((G+\delta G)^T B(\bX) (G+\delta G)\Big)^{-1}  \\
  &= (G^T A(\bX)G )(G^T B(\bX)G )^{-1} + 2(\delta G^T A(\bX) G )_\mathrm{sym}(G^T B(\bX)G )^{-1} \\
  &- 2(G^T A(\bX)G)(G^T B(\bX)G )^{-1}  (\delta G^T B(\bX) G )_\mathrm{sym} (G^T B(\bX)G )^{-1} 
  +\mathcal{O}(\|\delta G\|^2).
 \end{align*}
 Taking the expectation of the trace yields
 \begin{align*}
  &\mathcal{R}(G+\delta G) = \mathcal{R}(G) + \E\left[\trace\left( 2 \delta G^T A(\bX) G (G^T B(\bX)G )^{-1} \right)\right] \\
  &- \E\left[\trace\left(  2(G^T A(\bX)G)(G^T B(\bX)G )^{-1}  (\delta G^T B(\bX) G)  (G^T B(\bX)G )^{-1}   \right)\right] +\mathcal{O}(\|\delta G\|^2).
 \end{align*}
 Here we used the fact that $\trace(M_\mathrm{sym}S)=\trace(MS)$ holds for any square matrix $M$ and any symmetric matrix $S$. 
 Using the notation $\langle M,N\rangle = \trace(MN^T)$, we can write $\mathcal{R}(G+\delta G) = \mathcal{R}(G) + \langle\nabla\mathcal{R}(G),\delta G\rangle +\mathcal{O}(\|\delta G\|^2)$ where
 \begin{align}
  \nabla\mathcal{R}(G) &= 2 \E\left[ A(\bX)G(G^T B(\bX)G )^{-1} \right] \nonumber\\
  &-2\E\left[ B(\bX)G (G^T B(\bX)G )^{-1} G^T A(\bX)G(G^T B(\bX)G )^{-1} \right]. \nonumber
 \end{align}
 This shows that $\mathcal{R}(\cdot)$ is differentiable at $G$.
 Finally, the expression \eqref{eq:gradR} of $\nabla\mathcal{R}(G)$ is obtained by using the definitions of $H(G)$ and $\Sigma(G)$ (see \eqref{eq:H} and \eqref{eq:Sigma}) and by using the fact that $(S_1GS_2)_{\mathrm{vec}} = (S_2\otimes S_1)G_{\mathrm{vec}}$ for any symmetric matrices $S_1,S_2$.
 Both $H(G)$ and $\Sigma(G)$ are symmetric positive semidefinite, as the expectations of the Kronecker products of symmetric positive semidefinite matrices.

\end{document}